\documentclass[12pt]{amsart}

\usepackage{paralist}
\usepackage{amssymb}
\usepackage{stmaryrd}

\newtheorem{thm}{Theorem}
\newtheorem{prop}[thm]{Proposition}
\newtheorem{cor}[thm]{Corollary}
\newtheorem{lemma}[thm]{Lemma}
\theoremstyle{definition}

\newtheorem{example}[thm]{Example}
\theoremstyle{remark}
\newtheorem*{remark}{Remark}
\newtheorem*{ack}{Acknowledgement}

\newcommand{\N}{\mathbb{N}}

\newcommand{\R}{\mathbb{R}}
\newcommand{\C}{\mathbb{C}}
\newcommand{\bigoh}{\mathcal{O}}
\newcommand{\RE}{\operatorname{Re}}
\newcommand{\IM}{\operatorname{Im}}
\newcommand{\eps}{\varepsilon}
\newcommand{\Cinfty}{\ensuremath{C^{\infty}}}
\newcommand{\Ccinfty}{\ensuremath{C_c^{\infty}}}
\newcommand{\Dprime}{\ensuremath{C^{-\infty}}}
\newcommand{\setof}{\mathrel{;}}
\newcommand{\supp}{\operatorname{supp}}
\newcommand{\WF}{\operatorname{WF}}
\newcommand{\jap}[1]{\langle #1 \rangle}
\newcommand{\restrict}[1]{|_{#1}}
\newcommand{\Restrict}[1]{\big|_{#1}}

\newcommand{\diag}{\operatorname{diag}}
\newcommand{\Hom}{\operatorname{Hom}}
\newcommand{\End}{\operatorname{End}}
\newcommand{\Kernel}{\operatorname{ker}}
\newcommand{\divergence}{\operatorname{div}}
\newcommand{\Id}{\operatorname{Id}}
\newcommand{\dist}{\operatorname{dist}}
\newcommand{\inj}{\operatorname{inj}}
\newcommand{\tofrom}[2]{[#1\shortleftarrow #2]} 
\newcommand{\transport}{\tau} 
\newcommand{\CT}{\C T}
\newcommand{\horderiv}{\mspace{-1mu}\sideset{^h}{}{\mathop{\nabla}}\mspace{-1mu}}
\newcommand{\verderiv}{\mspace{-1mu}\sideset{^v}{}{\mathop{\nabla}}\mspace{-1mu}}
\newcommand{\Horderiv}[1]{\sideset{^h}{#1}{\mathop{\nabla}}}
\newcommand{\Verderiv}[1]{\sideset{^v}{#1}{\mathop{\nabla}}}
\newcommand{\intd}{\mspace{0.5mu}\operatorname{d}\mspace{-2.5mu}}
\newcommand{\dvol}{\intd V}
\newcommand{\xih}{\hat{\xi}}
\newcommand{\h}{h}  
\newcommand{\hdr}{D_{r/\h}}
\newcommand{\trace}{\operatorname{tr}}
\newcommand{\Lie}{\ensuremath{\mathcal{L}}}
\newcommand{\Def}{\operatorname{Def}}
\newcommand{\EllRegion}{\mathcal{E}}
\newcommand{\Op}{\operatorname{Op}}
\newcommand{\spec}{\operatorname{spec}}
\newcommand{\Sym}[1]{S^{#1}}
\newcommand{\Sykm}[2]{S^{#2,#1}}

\newcommand{\Sytm}[1]{S^{#1}_{\operatorname{tang}}}
\newcommand{\Sytkm}[2]{S^{#2,#1}_{\operatorname{tang}}}
\newcommand{\Sypm}[1]{S^{#1}_{\operatorname{pois}}}
\newcommand{\Sypkm}[2]{S^{#2,#1}_{\operatorname{pois}}}
\newcommand{\Psm}[1]{\Psi^{#1}}
\newcommand{\Psmphg}[1]{\Psi^{#1}_{\operatorname{phg}}}
\newcommand{\Pskm}[2]{\Psi^{#2,#1}}
\newcommand{\Pstkm}[2]{\Psi^{#2,#1}_{\operatorname{tang}}}
\newcommand{\halfdens}{\Omega^{1/2}}
\newcommand{\subprinc}[1]{#1_{\operatorname{sub}}}
\newcommand{\inu}{\iota_\nu}
\newcommand{\enu}{\epsilon_\nu}
\newcommand{\ueber}{\shortrightarrow}

\newcommand{\psdiff}{pseudo-differential}
\newcommand{\Psdiff}{Pseudo-differential}

\begin{document}

\title[Surface Quasimodes]{Rayleigh-Type Surface Quasimodes \\ in General Linear Elasticity}
\author[S. Hansen]{S\"onke Hansen}
\address{Institut f\"ur Mathematik\\ Universit\"at Paderborn\\ 33095 Paderborn\\ Germany}
\email{soenke@math.upb.de}
\urladdr{http://www.math.upb.de/~soenke/}
\date{\today\ (revised)}

\begin{abstract}
Rayleigh-type surface waves correspond to the characteristic variety,
in the elliptic boundary region, of the displacement-to-traction map.
In this paper, surface quasimodes are constructed
for the reduced elastic wave equation, anisotropic in general,
with traction-free boundary.
Assuming a global variant of a condition of Barnett and Lothe,
the construction is reduced to an eigenvalue problem for
a selfadjoint scalar first order \psdiff\ operator on the boundary.
The principal and the subprincipal symbol of this operator are computed.
The formula for the subprincipal symbol seems to be new even in the isotropic case.
\end{abstract}

\maketitle

\section{Introduction}
Rayleigh~\cite{rayleigh:87:surface} discovered the existence of surface waves
which propagate along a traction-free flat boundary of an isotropic elastic body
and which decay exponentially into the interior.
The propagation speed of the surface wave is strictly less than that of body waves.
Barnett and Lothe~\cite{lothe/barnett:76:exist-surface-wave}
showed that Rayleigh-type surface waves can also exist at flat boundaries
of anisotropic elastic media.

The goal of this paper is to construct, for elastic media which are not
necessarily isotropic, Rayleigh-type surface quasimodes which 
are asymptotic to eigenvalues or resonances.
We use a geometric version of semiclassical microlocal analysis.

The  Rayleigh wave phemomenon of isotropic elastodynamics was explained
by Taylor~\cite{taylor:79:rayleigh-waves} as propagation of singularities,
over the elliptic boundary region, for the
Neumann (displacement-to-traction) operator.
Nakamura~\cite{nakamura:91:rayleigh-pulses} generalized this to anisotropic media,
using the theory of Barnett and Lothe.
Assuming isotropy of the elastic medium,
Cardoso-Popov~\cite{CardosoPopov92RayleighQM} and Stefanov~\cite{Stefanov00LowerBd}
constructed Rayleigh quasimodes.

Let $(M,g)$ be an oriented Riemannian manifold with 
non-empty compact smooth boundary $X$.
The (infinitesimal) displacement of an elastic medium
occupying $M$ is a vector field $u$ on $M$.
The Lie derivative of the metric tensor is a symmetric tensor field,
$\Def u=\Lie_u g/2$,
called the deformation (strain) tensor caused by the displacement $u$.
The elastic properties are defined by the elasticity (stiffness) tensor.
This is a real fourth order tensor field $C\in\Cinfty(M;\End(T^{0,2}M))$, $e\mapsto Ce$,
which maps into symmetric tensors and vanishes on antisymmetric tensors.
We assume positive definiteness of $C$, i.e.,
$(e|f)_C=(Ce|f)$ defines an inner product on the space of symmetric tensors $e$ and $f$.
Here $(\cdot|\cdot)$ denotes the inner product on tensors induced from $g$.
This assumption is often called the strong convexity condition.
If coordinates are given, then the components of $C$ satisfy
symmetries, $C^{ijk\ell}= C^{jik\ell}= C^{k\ell ij}$,
and $C^{ijk\ell}e_{ij}e_{k\ell}>0$ if $e_{ij}$ is a nonzero symmetric tensor.
(We use the summation convention.)
Denote the Riemannian volume elements on $M$ and $X$ by $\dvol_M$ and $\dvol_X$, respectively.
The elasticity operator $L$ and the traction $T$ are defined,
on compactly supported vector fields, by
\begin{equation}
\label{def-elast-traction}
\int_M (\Def u|\Def v)_C\dvol_M = \int_M (L u|v) \dvol_M + \int_X (Tu|v)\dvol_X.
\end{equation}
A positive mass density $\rho\in\Cinfty(M)$ and the elasticity
tensor $C$ define the material properties of the elastic medium.
If the surface $X$ is traction-free, then vibrations of the medium 
are solutions of the following eigenvalue problem:
$Lu=\lambda^{2}\rho u$ in $M$, $Tu=0$ at $X$.
See \cite{marsden/hughes:83:elasticity} for linear elasticity
in the language of Riemannian geometry.

The principal symbol of $L$, and of the $h$-differential operator $\h^2 L$,
equals the acoustic tensor, $c(\xi)=c(\xi,\xi)\in\End(\CT_x M)$, $\xi\in T_x^* M$;
see \eqref{symb-L}.
Here the associated acoustical tensor $c(\xi,\eta)\in\End(\CT_x M)$,
$\xi,\eta\in T_x^*M$, is defined as follows:
\begin{equation}
\label{C-and-c}
\big(c(\xi,\eta)v|w\big) =\big(v\otimes \eta|w\otimes\xi\big)_C.
\end{equation}
(Using $g$, we identify vectors with covectors.)
The $ik$-th covariant component of $c(\xi,\eta)$ equals $C^{ijk\ell}\xi_j\eta_\ell$.

The existence of Rayleigh waves depends on the characteristic variety, $\Sigma$,
of the surface impedance tensor, $z$.
To define $z$,
we first recall the definition of the elliptic boundary region, $\EllRegion\subset T^*X$.
Let $\nu$ denote the unit exterior conormal field of the boundary $X$.
Identify $T^*X=\nu^\perp\subset T_X^* M$.
By definition,
$\xi\in\EllRegion$ iff $c(\xi+s\nu)-\rho\Id$ is positive definite for real $s$.
From the factorization theory of selfadjoint matrix polynomials one gets
$q(\xi)\in\End(\CT_x M)$, $\xi\in\EllRegion\cap T_x^* X$, such that
\begin{equation}
\label{spectral-factorization}
c(\xi+s\nu)-\rho\Id =\big(s\Id-q^*(\xi)\big) c(\nu) \big(s\Id-q(\xi)\big),
\end{equation}
$s\in\C$.
Moreover, the spectrum of $q(\xi)$ lies in the lower half-plane, $\spec q(\xi)\subset\C_-$,
and these properties determine $q(\xi)$ uniquely.
The surface impedance tensor $z$ is defined as follows:
\begin{equation}
\label{def-surfimped-z}
z(\xi)= i c(\nu) q(\xi) +i c(\nu,\xi),\quad \xi\in\EllRegion.
\end{equation}
The significance of $z$ results from the fact, proved in Lemma~\ref{Z-op-mainsymb},
that $z$ is the principal symbol of a parametrix 
of the displacement-to-traction operator.
In physics, the meaning of $z$ is that it relates the amplitudes of displacements
to the amplitudes of tractions (forces) needed to sustain these.

The surface impedance tensor is Hermitian, and positive definite for large $|\xi|$,
\cite[Theorem~6]{lothe/barnett:85:surface-wave-impedance-method}.
If $\dim M=3$, then
\begin{gather}
\tag{U} \label{hypo-atmost}
\text{$z(\xi)$, $\xi\in\EllRegion$, has at most one non-positive eigenvalue.}
\end{gather}
This property expresses the uniqueness of Rayleigh-type surface waves,
\cite[Theorem~8]{lothe/barnett:85:surface-wave-impedance-method}.
In case $\dim X\neq 3$, we shall assume \eqref{hypo-atmost} as a hypothesis.
The characteristic variety of $z$,
\[
\Sigma=\{\xi\in\EllRegion\setof \det z(\xi)=0\},
\]
is a smooth hypersurface, transversal to the radial directions of the fibers of $T^*X$.
Compare \cite[Theorem~7]{lothe/barnett:85:surface-wave-impedance-method}.
Rayleigh waves exist only if $\Sigma$ is not empty.
We shall make the stronger assumption that $\Sigma$ intersects every radial line:
\begin{gather}
\tag{E1} \label{hypo-intersect}
\Sigma\cap \R_+\xi\neq\emptyset\;
\text{ if }\; \xi\in T^*X\setminus 0.
\end{gather}
Compare \cite[Theorem~12]{lothe/barnett:85:surface-wave-impedance-method},
\cite[Theorem 2.2]{nakamura:91:rayleigh-pulses}, \cite[(ERW)]{KawaNaka00Poles}.
Assuming \eqref{hypo-atmost} and \eqref{hypo-intersect},
there exists a unique $p\in\Cinfty(T^*X\setminus 0)$, $p>0$,
homogeneous of degree $1$, such that
\begin{equation}
\label{Sigma-eq-pinv1}
\Sigma=p^{-1}(1).
\end{equation}
See Proposition~\ref{prop-radial-line-and-charvar}.
Furthermore, the kernel of $z$ defines a line bundle,
$\Kernel z\ueber\Sigma$, over the compact base $\Sigma$.
We shall require that its first Chern class vanishes:
\begin{gather}
\tag{E2} \label{hypo-trivial}
\Kernel z\ueber\Sigma \text{ is a trivial line bundle.}
\end{gather}
In particular, the bundle is assumed to possess a unit section.
Property \eqref{hypo-trivial} is stable with respect to homotopies in
the material properties; see Corollary~\ref{cor-homotopy}.
In the case of isotropic elasticity with positive Lam\'e parameters,
\eqref{hypo-atmost}, \eqref{hypo-intersect}, and \eqref{hypo-trivial} hold.
Moreover,
\[
\Sigma=\{c_r|\xi|=1\}\subset\EllRegion=\{c_s|\xi|>1\}, \quad p(\xi)=c_r|\xi|.
\]
Here $c_r$ is the propagation speed of the Rayleigh surface wave which
is strictly less than the speeds of the body waves, $0<c_r<c_s<c_p$.
See Example~\ref{exa-iso-1}.

Next we state the central result of this paper:
The traction-free surface eigenvalue problem can be intertwined
with a selfadjoint eigenvalue problem on the boundary.
We employ a semiclassical \psdiff\ calculus, with
distributions and operators depending on a small parameter, $0<\h\leq 1$.
We write $A_\h\equiv B_\h$ iff the Schwartz kernel of $A_\h -B_\h$
belongs to $\Cinfty$ with seminorms satisfying $\bigoh_{\Cinfty}(\h^\infty)$.

\begin{thm}
\label{thm-reduct}
Assume $\dim M=3$, or \eqref{hypo-atmost}.
Assume \eqref{hypo-intersect}, \eqref{hypo-trivial}.
Given a unit section $v$ of $\Kernel z\ueber \Sigma$, 
there exists a selfadjoint, elliptic operator
$P\in\Psm{1}(X;\halfdens)$, independent of $\h$,
and operators,
\begin{gather*}
B_\h:L^2(X;\CT_X M)\to L^2(M;\CT M),
\quad \|B_\h\|=\bigoh(\h^{1/2}), \\
J_\h,\tilde{J}_\h\in\Pskm{0}{0}(X;\halfdens,\CT_X M),
\quad \text{$J_\h^* J_\h$ elliptic at $\Sigma$,}
\end{gather*}
such that
\begin{equation*}
\big(\h^2 L-\rho\big) B_\h\equiv 0, \quad
TB_\h J_\h \equiv \tilde{J}_\h(P-\h^{-1}),
\end{equation*}
and $B_\h\restrict{X}=\Id$ in a neighbourhood of $\Sigma$.
The principal symbol of $P$ equals $p$ of \eqref{Sigma-eq-pinv1}.
Furthermore, there is a formula, \eqref{subprinc-P}, for the
subprincipal symbol $\subprinc{p}$ of $P$.
If $v$ is changed to another unit section, $e^{i\varphi}v$, then
the subprincipal symbol changes to $\subprinc{p} + \{p,\varphi\}$,
where $\{p,\varphi\}$ denotes the Poisson bracket.
\end{thm}

This result is known in the isotropic case,
\cite{CardosoPopov92RayleighQM}, \cite{Stefanov00LowerBd},
except for the assertions about the subprincipal symbol.

The operator $B_\h$ is a parametrix of the Dirichlet problem near $\Sigma$;
see Proposition~\ref{exist-dirichlet-param}.
Its range consists of functions which are smooth in the interior of $M$,
supported in a preassigned neighbourhood of the boundary,
and which decay like $e^{-\delta\dist_X/\h}$ into the interior.

Ignoring finitely many eigenvalues the spectrum of $P$ consists
of a sequence of positive eigenvalues $\mu_j\uparrow\infty$.
Applying Theorem~\ref{thm-reduct} to an associated orthonormal system
of eigenvectors we obtain, in Proposition~\ref{prop-qm-indep},
a sequence of quasimode states:
$L u_j-\mu_j^{2} \rho u_j =\bigoh_{\Cinfty}(\h_j^{\infty})$
with boundary tractions equal to zero.
Moreover, the quasimode states are well-separated.
The construction also works when starting with a sequence
of almost orthogonal quasimode states of $P$.

The unbounded operator
$D\to L^2(M;\CT M;\rho\dvol_M)$, $u\mapsto \rho^{-1}Lu$,
with domain $D=\{u\in \Ccinfty\setof Tu=0\}$,
is symmetric and nonnegative.
The associated quadratic form is
given by the left-hand side of \eqref{def-elast-traction}.
Denote $L_T$ the Friedrichs extension of this operator.
For a selfadjoint operator $A$ with spectrum consisting of a
sequence of eigenvalues accumulating at $+\infty$, 
denote $N_A(\lambda)$ the usual counting function for the eigenvalues of $A$.
The following lower bound on $N_{L_T}(\lambda)$ is an example
application of our results.

\begin{cor}
\label{cor-ev-asymp}
Assume $M$ compact, $\dim M=3$, and \eqref{hypo-intersect}, \eqref{hypo-trivial}.
Let $P$ be the selfadjoint operator given in Theorem~\ref{thm-reduct}.
For every $m>1$,
$N_{L_T}(\lambda)- N_{P}(\lambda-\lambda^{-m})$ is bounded from below.
\end{cor}

Rayleigh waves have been studied in several papers with the emphasis
of getting information about resonances in scattering theory,
\cite{StefanovVodev95duke}, \cite{StefanovVodev96Reson},
\cite{SjostrandVodev97RayleighRes}, \cite{Stefanov00LowerBd},
and, for anisotropic media,
\cite{KawaNaka00Poles}.
Stefanov \cite{Stefanov00LowerBd} 
uses Rayleigh quasimodes to derive lower bounds on the number of resonances.
See the remark at the end of section~\ref{sect-modes}
about going from quasimodes to resonances.

The subprincipal symbol $\subprinc{p}$ affects the eigenvalue asymptotics of $P$,
\cite{duistermaat/guillemin:75:spectrum},
and it enters quasimode constructions, \cite{CardosoPopov92RayleighQM}.
The subprincipal symbol occurs in the final formulas via integrals,
such as $\int_{S^*X} \subprinc{p}$ and $\int_\gamma \subprinc{p}$,
where $\gamma$ is a closed bicharacteristic.
We point out that these integrals do not depend on the choice of the unit section $v$
in Theorem~\ref{thm-reduct}, although $\subprinc{p}$ itself does.
An important aim of the present work is to give explicit formulas for the
subprincipal symbol of $P$.
These seem to be new even in the isotropic case which is dealt with
in more detail in Proposition~\ref{prop-subprinc-isoelast}.
The main difficulty comes from the fact that an invariant notion
of subprincipal symbol has only been available for scalar operators.
To overcome this obstacle we adapt and systematically use the geometric
\psdiff\ calculus of Sharafutdinov~\cite{Sharaf04geo1,Sharaf05geo2}
which assumes given a differential geometric structure.
The principal and subprincipal symbol levels are contained in
the leading symbol of a (pseudo-)differential operator.

The paper is organized as follows.
In section~\ref{section-impedance} the surface impedance tensor is studied;
in particular, a selfcontained treatment of Barnett-Lothe theory is given.
The leading geometric symbols of some differential operators
are computed in section~\ref{sect-conn-geosymb}.
In section~\ref{sect-collar} we geometrically decompose the elasticity operator near the boundary
into normal and tangential operators, keeping track of leading geometric symbols.
Section~\ref{sect-factorization} gives, microlocally at the elliptic region $\EllRegion$,
a factorization of $\h^2L-\rho$ into a product of first order operators.
Using the factorization, we construct in section~\ref{sect-parametrix}
a parametrix for the Dirichlet problem microlocally at $\EllRegion$.
The displacement-to-traction operator $Z$ is defined in section~\ref{sect-dt-op},
and its leading geometric symbol is determined.
In section~\ref{sect-diag} we derive a diagonalization of $Z$,
and we prove Theorem~\ref{thm-reduct}.
In section~\ref{sect-modes} we construct localized traction-free surface quasimodes,
and we prove Corollary~\ref{cor-ev-asymp}.
In section~\ref{sect-isosubpr} we calculate, for an isotropic elastic medium,
the subprincipal symbol of $P$.
The appendix~\ref{app-geopsdo-calc} contains a detailed exposition
of Sharafutdinov's geometric \psdiff\ calculus in a semiclassical setting.

\begin{ack}
The author thanks G.~Mendoza for inspiring discussions about an earlier version
of the present work, and for helpful remarks,
in particular, concerning condition~\eqref{hypo-trivial}.
Thanks go also to P.~Stefanov for helpful responses and questions via email.
\end{ack}

\section{The Surface Impedance Tensor}
\label{section-impedance}

First we collect some well-known facts about spectral factorizations of selfadjoint matrix polynomials.
Refer to \cite[Chapter~11]{gohberg/lancaster/rodman:82:matrixpolynomials}.
Let $V$ be a finite-dimensional complex Hilbert space,
and $f(s)=as^2+bs+c\in \End(V)$ a quadratic polynomial in the complex variable $s$.
The spectrum of $f$ is the set of $s\in\C$ such that $\Kernel f(s)\neq 0$.
Assume that the leading coefficient of $f$, $a$, is nonsingular.
Then the spectrum is finite.
Assume that $f$ is selfadjoint, $f(s)^*=f(\bar{s})$, and that, in addition
$f(s)$ is positive definite for real $s$.
The spectrum of $f$ is a disjoint union $\sigma_+\cup\sigma_-$,
where $\sigma_+$ and $\sigma_-$ are contained in the upper and lower half-planes, respectively.
There is a unique $q\in\End(V)$ such that
$f(s)=(s-q^*)a(s-q)$, and the spectrum of $q$ equals $\sigma_-$.
If $\gamma$ is a closed Jordan curve which contains $\sigma_-$ in its interior
and $\sigma_+$ in its exterior, then
\begin{equation}
\label{q-int-rep}
q \oint_{\gamma} f(s)^{-1}\intd s
    = \oint_{\gamma} s f(s)^{-1}\intd s.
\end{equation}
The integral on the left is nonsingular.
Jordan-Keldysh chains are a means to compute $q$.
In particular, one has $qv=sv$ if $f(s)v=0$ and $\IM s<0$.
Moreover, the solvency equation $f(q)=0$ holds.

The following representation of the factor $q$ by integrals is important.
We shall also apply it later to establish symbol properties.
Denote $i=\sqrt{-1}$ the imaginary unit.

\begin{lemma}
\label{factor-sapolyn}
Let $f$ and $q$ be as above. Then
\begin{equation}
\label{factor-sapolyn-rep}
a\,q f_0 = -\pi i\Id + f_1,
\end{equation}
where 
$f_0 = \int_{-\infty}^\infty f(s)^{-1}\intd s$ is
selfadjoint and positive definite, and
\[
f_1 = \int_{|s|\leq 1} s a f(s)^{-1}\intd s
      + \int_{|s|> 1} s^{-1} \big(s^2 a -f(s)\big) f(s)^{-1}\intd s.
\]
The integrals converge absolutely in $\End(V)$.
\end{lemma}

\begin{proof}
Let $\gamma_R$ denote the negatively oriented closed contour composed
of the semicircle $\{|s|=R,\IM s\leq 0\}$
and the interval $[-R,R]$.
The integral representation~\eqref{q-int-rep} holds with $\gamma=\gamma_R$
if $R$ is sufficiently large.
We have $f(s)^{-1} = s^{-2} a^{-1}+\bigoh(|s|^{-3})$ as $|s|\to\infty$.
It follows that
\[
\lim_{R\to\infty} \oint_{\gamma_R} f(s)^{-1}\intd s 
= \int_{-\infty}^\infty f(s)^{-1}\intd s,
\]
and
\[
\lim_{R\to\infty}\oint_{\gamma_R}sa f(s)^{-1}\intd s 
= -\pi i\Id + \lim_{R\to\infty} \int_{-R}^R sa f(s)^{-1} \intd s.
\]
Using $s^2 a f(s)^{-1}-\Id =(s^2 a-f(s)) f(s)^{-1}$ we obtain
\[
\int_{1<|s|\leq R} sa f(s)^{-1} \intd s
    = \int_{1<|s|\leq R} s^{-1} \big(s^2 a -f(s)\big) f(s)^{-1} \intd s.
\]
This proves the formulas.
The remaining assertions follow from these
and the positive definiteness of $f(s)$.
\end{proof}

Let $\xi\in T_x^* X$, and denote $\nu\in T_x^*M$ the unit exterior normal.
Set $a=c(\nu)$, $a_1(\xi)=c(\nu,\xi)$, and $a_2(\xi)=c(\xi)$.
Note that $a_1(\xi)^*=c(\xi,\nu)$.
The polynomial
\begin{equation}
\label{def-f-of-s}
f(s)=c(\xi+s\nu)-\rho =as^2 +(a_1+a_1^*)s+a_2 -\rho,
\end{equation}
$f(s)=f(s,\xi)$,
has values in $\End(\CT_x M)$.
It is selfadjoint with real coefficients.
By definition, $\xi\in\EllRegion$ iff $f(s)$ is positive definite for $s\in\R$.
\begin{lemma}
The elliptic region $\EllRegion$ is an open subset of $T^*X$ with compact complement.
Moreover, $\EllRegion$ is symmetric and star shaped with respect to infinity,
i.e., $t\xi\in\EllRegion$ whenever $\xi\in\EllRegion$ and $t$ real, $|t|\geq 1$.
\end{lemma}
\begin{proof}
By positive definiteness of $C$, there exists $\delta>0$ such that
$g(v,c(\eta)v)\geq \delta|v\otimes\eta+\eta\otimes v|^2$ for (co-)vectors $v,\eta$.
The symmetrization of a non-zero real elementary tensor is non-zero.
Therefore, with a new $\delta>0$,
in the sense of selfadjoint maps,
\begin{equation}
\label{C-is-elliptic}
c(\eta)\geq \delta|\eta|^2\Id. 
\end{equation}
Since $|\xi+s\nu|^2=|\xi|^2+s^2$ the first assertions follow.
The symmetry and the star-shapedness follow from $c(t\eta)=t^2 c(\eta)$.
\end{proof}

If $\xi\in\EllRegion$, then
\eqref{factor-sapolyn-rep} holds with $q=q(\xi)$, $f_j=f_j(\xi)$.
The spectral factor $q$ solves \eqref{spectral-factorization};
using current notation:
\begin{equation}
\label{factor-with-q}
as^2 +(a_1+a_1^*)s + a_2 -\rho = (s - q^*) a(s-q).
\end{equation}
The spectrum of $q$ lies in the lower halfplane,
and $q$ is uniquely determined by these properties.
Notice that $q$ is a smooth section of the bundle $\pi^* \End(\CT_X M)\ueber\EllRegion$,
where $\pi:\EllRegion\subset T^*X\ueber X$ denotes the canonical projection.

The surface impedance tensor, defined in \eqref{def-surfimped-z},
equals $z=i(aq+a_1)$.
Lemma~\ref{factor-sapolyn} implies
\begin{equation}
\label{barnett-lothe-identity}
z f_0 = \pi \Id + i(f_1+a_1 f_0).
\end{equation}
Since the $f_j$'s are real, this gives the decomposition of $z$
into real and imaginary parts.
Following \cite{mielke/fu:04:uniq-surface-wave-speed},
we shall use the Ricatti-type equation
\begin{equation}
\label{ricatti-z}
(z+ia_1^*)a^{-1}(z-ia_1) = a_2-\rho
\end{equation}
to deduce properties of $z$.
Equation~\eqref{ricatti-z} follows upon insertion of $q=-a^{-1}(iz+a_1)$
into the solvency equation associated with \eqref{factor-with-q},
\begin{equation}
\label{solvency}
aq^2 +(a_1+a_1^*)q + a_2-\rho = 0.
\end{equation}
A consequence of \eqref{ricatti-z} is
\begin{equation}
\label{ricatti-z-prime}
(iq)^* z'+z'(iq) = {a_1^*}'q +q^* a_1' +(a_2-\rho)' +q^*a' q,
\end{equation}
where the prime denotes the derivative with respect to some chosen parameter.
The spectra of $q$ and $q^*$ are disjoint.
Therefore, the Sylvester equation $(iq)^*x+x(iq)=i(xq-q^*x)=y$
has a unique solution $x$ for given $y$.
The solution is, in fact, given by an integral,
$x=\int_{-\infty}^0 \exp(irq)^* y\exp(irq)\intd r$.
It follows that $x$ is positive definite if $y$ is.

\begin{prop}
\label{prop-z}
The impedance tensor $z(\xi)$, $\xi\in\EllRegion$, has the following properties.
\begin{compactenum}[(i)]
\item\label{z-selfadjoint}
  $z(\xi)$ is selfadjoint.
\item\label{z-posdef-large}
  $z(\xi)$ is positive definite if $|\xi|$ is sufficiently large.
\item\label{z-re-posdef}
  $\RE z(\xi)$ is positive definite.
\item\label{z-two-eigen}
  $z(\xi)$ has at least two positive eigenvalues if $\dim M\geq 3$.
\item\label{z-radderiv-posdef}
  $(d/dt)\restrict{t=1} t^{-1} z(t\xi)=\dot{z}-z$ is positive definite.
\item\label{z-conjugate}
  The complex conjugate $\overline{z(\xi)}= z(-\xi)$.
\end{compactenum}
\end{prop}
We call $\dot{z}(\xi)=(d/dt)\restrict{t=1} z(t\xi)$ the radial derivative of $z$ at $\xi$.
It follows from \eqref{z-radderiv-posdef} that $\dot{z}$ is positive definite
on the kernel of $z$, $\Kernel z$.
\begin{proof}
To prove \eqref{z-selfadjoint} we follow the arguments in
\cite[Theorem~2.2]{mielke/fu:04:uniq-surface-wave-speed}.
First note that \eqref{ricatti-z} remains true if $z$ is replaced by $z^*$.
Subtracting the two equations we get the Sylvester equation
$(iq)^*(z-z^*)+(z-z^*)(iq)=0$, implying $z-z^*=0$.

It follows from \eqref{barnett-lothe-identity} that $\RE z=\pi f_0^{-1}$.
This proves \eqref{z-re-posdef}.

Suppose $\dim M\geq 3$.
Aiming at an indirect proof of \eqref{z-two-eigen},
assume that $z(\xi)$, $\xi\in T_x^*X$, has at most one positive eigenvalue.
Then there exists $w\in \CT_x M$ such that $z(\xi)$ is
negative semidefinite on the orthogonal complement $w^\perp$.
Choose a real vector $v\neq 0$ which is orthogonal to both $\RE w$ and $\IM w$.
Then $v\in w^\perp$, and
$(\RE z(\xi)v|v)=(z(\xi)v|v)\leq 0$, contradicting the positive definiteness of $\RE z$.

Next we prove \eqref{z-radderiv-posdef} following the method 
of \cite[Theorem~2.3]{mielke/fu:04:uniq-surface-wave-speed}.
Since $a_j(\xi)$ is homogeneous of degree $j$ in $\xi$,
equation \eqref{ricatti-z} implies,
\[
(t^{-1}z(t\xi)+ia_1^*(\xi))a^{-1}(t^{-1}z(t\xi)-ia_1(\xi)) = a_2(\xi)-t^{-2}\rho.
\]
Taking the derivative with respect to $t$ at $t=1$, we get
\[
(iq)^*(\dot{z}-z)+(\dot{z}-z)(iq) =2\rho.
\]
By the remarks following \eqref{ricatti-z-prime},
we see that $\dot{z}-z$ is positive definite.

We now prove \eqref{z-conjugate}.
Note $f(s,-\xi)=f(-s,\xi)$, $f_j(-\xi)=(-1)^j f_j(\xi)$, and $a_1(-\xi)=-a_1(\xi)$.
Using \eqref{barnett-lothe-identity} we derive
$z(-\xi)f_0(\xi)=\overline{z(\xi)f_0(\xi)}$.
Since $f_0$ is real and nonsingular the formula follows.

It remains to prove \eqref{z-posdef-large}.
Let $\eta\in T_xX$, $|\eta|=1$.
It suffices to show that $z_\infty=\lim_{t\uparrow\infty} t^{-1}z(t\eta)$ exists and is positive definite.
Set $q_t=t^{-1}q(t\eta)$, $t>1$ large.
From \eqref{factor-with-q} deduce
\[
as^2 +(a_1(\eta)+a_1^*(\eta))s + a_2(\eta) -t^{-2}\rho
  = (s - q^*_t) a(s-q_t),
\quad s\in\R.
\]
Using \eqref{factor-sapolyn-rep} and dominated convergence in the integrals
giving $f_j$ we infer that $q_\infty=\lim_{t\uparrow\infty} q_t$ exists.
In particular, $t^{-1}z(t\eta)$ converges to
$z_\infty=i(a q_\infty+a_1(\eta))$ as $t\uparrow\infty$.
Let $y\in T_xM$ such that $(z_\infty y|y)\leq 0$.
We must show $y=0$.
Set $w(r)=\exp(ir q_\infty)y$, $r\leq 0$.
The solvency equation \eqref{solvency} holds with $q$ replaced by $q_\infty$, $\rho=0$.
Therefore,
$aD_r^2w+(a_1+a_1^*)D_r w+a_2w=0$ holds,
where we use the abbreviation $a_j=a_j(\eta)$.
Take the inner product in $\CT_x M$ with $w$ and integrate.
A partial integration gives
\begin{align*}
\int_{-\infty}^0 & (a D_r w+a_1 w|D_r w) +(D_r w| a_1 w)+(a_2 w|w)\intd r \\
    &= i(aD_r w+a_1w|w)\big|_{-\infty}^0 = (z_\infty y|y) \leq 0.
\end{align*}
Set $W(r)=w(r)\otimes\eta+D_r w(r)\otimes\nu\in \End(\CT_x M)$.
Recall $a=c(\nu)$, $a_1=c(\nu,\eta)$, $a_2=c(\eta)$, and \eqref{C-and-c}.
We have shown:
\[
\int_{-\infty}^0 (W(r)|W(r))_C \intd r \leq 0.
\]
Recall that $C$ is real, and that $(\;|\;)_C$ is
an inner product on symmetric $2$-tensors.
It follows that the symmetrization of $W(r)$ vanishes for all $r\leq 0$.
In particular,
\begin{equation}
\label{eq-for-w-of-r}
\big(w(r)\otimes\eta +D_r w(r)\otimes\nu\mid \zeta\otimes\nu+\nu\otimes\zeta\big)=0
\end{equation}
for $\zeta\in \CT_x^*M$, $r\leq 0$.
Recall $(\eta|\nu)=0$.
Setting $\zeta=\nu$, we derive $D_r(w(r)|\nu)=(D_r w(r)|\nu)=0$.
Since $w(r)\to 0$ as $r\to -\infty$, we obtain $(w(r)|\nu)=0$.
Now, \eqref{eq-for-w-of-r} simplifies to
$(D_r w(r)|\zeta)=0$.
Since $\zeta$ is arbitrary, this implies, successively, $D_r w=0$, $w=0$, $y=0$.
\end{proof}

If $\dim M=3$ then \eqref{hypo-atmost} holds.
This follows from \eqref{z-two-eigen}.

\begin{prop}
\label{prop-radial-line-and-charvar}
Assume \eqref{hypo-atmost}.
Then the characteristic variety of $z$,
$\Sigma=\{\det z(\xi)=0\}$, is a smooth hypersurface in $\EllRegion$.
Each radial line $\R_+\xi\subset T^*X$ intersects $\Sigma$ in at most one point,
and the intersection is transversal.
The kernel of $z\restrict{\Sigma}$ defines a line bundle $\Kernel z\ueber \Sigma$.
Assume, in addition, \eqref{hypo-intersect}.
There is a unique $p\in \Cinfty(T^*X\setminus 0)$,
homogeneous of degree one, such that $\Sigma=p^{-1}(1)$.
Moreover, $p>0$, and $p(-\xi)=p(\xi)$.
\end{prop}
\begin{proof}
From the assumption and \eqref{z-radderiv-posdef} of Proposition~\ref{prop-z}
it follows that $(d/dt) \det z(t\xi)>0$ if $t\xi\in\Sigma$, $t>0$.
In particular, zero is a regular value of $\det z$.
Hence $\Sigma$ is a codimension one submanifold transversal to the radial field.
Since $\R_+\xi\cap\EllRegion$ is connected, a given radial line $\R_+\xi$
intersects $\Sigma$ in at most one point.
Because of \eqref{hypo-atmost} and the selfadjointness of $z$,
zero is simple eigenvalue of $z$.
It follows that $\Kernel z\ueber\Sigma$ is a line bundle.
Now assume also \eqref{hypo-intersect}.
Then each radial line intersects $\Sigma$ in a unique point.
Define $p$ as follows.
For $0\neq\xi\in T^*X$ set $p(\xi)=1/t$ if $t\xi\in\Sigma$, $t>0$.
Smoothness of $p$ follows from the implicit function theorem.
The evenness of $p$ is a consequence of \eqref{z-conjugate}.
The other properties of $p$ are obvious.
Clearly, the homogeneity and $p\restrict{\Sigma}=1$ determine $p$ uniquely.
\end{proof}

\begin{cor}
\label{cor-homotopy}
Let $\rho_t$ and $C_t$ be homotopies of the mass densities
and the elasticity tensors, $0\leq t\leq 1$.
Assume that the associated surface impedance tensors $z_t$
and their characteristic varieties $\Sigma_t$
satisfy \eqref{hypo-atmost} and \eqref{hypo-intersect} for every $t$.
The line bundles $\Kernel z_0\ueber \Sigma_0$
and $\Kernel z_1\ueber \Sigma_1$ are isomorphic.
\end{cor}
\begin{proof}
The factorization \eqref{spectral-factorization} and the definition of the
impedance tensor imply that $z_t$ depends continuously on the homotopy parameter $t$.
It follows from Proposition~\ref{prop-radial-line-and-charvar} that the
characteristic varieties are canonically diffeomorphic to the sphere bundle $SX$.
We deduce that the Chern classes of the bundles $\Kernel z_t\ueber SX$ do not depend on $t$.
The assertion follows from this.
\end{proof}

\begin{example}
\label{exa-iso-1}
We consider, as special case, an isotropic elastic medium.
We shall verify \eqref{hypo-atmost}, \eqref{hypo-intersect}, and \eqref{hypo-trivial}.
The elasticity tensor reads, in component notation,
\begin{equation}
\label{C-iso-elast}
C^{ijk\ell} = \lambda g^{ij} g^{k\ell} + \mu(g^{ik} g^{j\ell} + g^{i\ell} g^{jk}),
\end{equation}
where $\lambda, \mu$ denote the Lam\'e parameters.
Equivalently,
\begin{equation}
\label{c-iso}
c(\xi,\eta) = \lambda \xi\otimes \eta +\mu\eta\otimes \xi +\mu g(\xi,\eta)\Id.
\end{equation}
Positive definiteness of $C$ is equivalent to $\mu>0$, $\lambda\dim M+2\mu>0$.
We make the stronger assumption $\lambda,\mu>0$.
Let $\xi\in T_x^*X$.
We list the eigenvalues $s\in\C$ and the eigenvectors $v\in\CT_x M$
of the quadratic polynomial $c(\xi+s\nu)-\rho$:
\begin{compactenum}[(a)]
\item $(\lambda+2\mu)(|\xi|^2+s^2)-\rho=0$ and $v=\xi+s\nu$,
\item $\mu(|\xi|^2+s^2)-\rho=0$ and $v=s\xi-|\xi|^2\nu$,
\item $\mu(|\xi|^2+s^2)-\rho=0$ and $v$ is orthogonal to $\xi$ and $\nu$.
\end{compactenum}
Introduce $c_p=\sqrt{(\lambda+2\mu)/\rho}$ and $c_s=\sqrt{\mu/\rho}$,
the speeds of pressure and of shear waves, respectively.
Assume that $\xi\in\EllRegion$.
This is equivalent to $c_s|\xi|>1$.
The above eigenvalues and eigenvectors diagonalize $q$, $q(\xi)v=sv$ if $\IM s <0$.
Denote $V$ the subbundle of $\pi^*\big(\CT_X M\big)\ueber \EllRegion$
spanned by $\nu$ and $\xi$, and $V^\perp$ its orthogonal bundle.
Fix the orthonormal frame $\nu, \xih =\xi/|\xi|$ of $V$,
and choose an orthonormal frame of $V^\perp$.
In block decompositions of matrices we let
the indices $1$ and $2$ correspond to $V$ and $V^\perp$, respectively.
We denote $(e)_{ij}$ the block $ij$ of the matrix which represents the endomorphism $e$.
Observe that $q$ leaves $V$ and $V^\perp$ invariant, $(q)_{12}=0=(q)_{21}$.
A simple computation gives
\begin{equation}
\label{iq-iso-block}
(i q)_{11} = \frac{|\xi|}{b}
\left[\begin{array}{cc}
ut\sqrt{1-t} & -i(b-ut)\\
i(b-t) & t\sqrt{1-ut}
\end{array}\right].
\end{equation}
Here $t = (c_s |\xi|)^{-2}$, $u=(c_s/c_p)^2=\mu/(\lambda+2\mu)$,
$b=1-\sqrt{1-ut}\sqrt{1-t}$.
Moreover, $(i q)_{22}$ equals $|\xi|\sqrt{1-t}$ times the unit matrix.
The maps $a=c(\nu)$ and $a_1=c(\nu,\xi)$ also leave $V$ and $V^\perp$ invariant.
We compute
\begin{equation}
\label{z-iso-block}
(z)_{11}= \frac{\mu|\xi|}{b}
\left[\begin{array}{cc}
t\sqrt{1-t} & -i(2b-t)\\
i(2b-t) & t\sqrt{1-ut}
\end{array}\right],
\end{equation}
and $(z)_{22}=\mu (iq)_{22}$.
The determinant of $z$ equals $(\mu|\xi|\sqrt{1-t})^{\dim V^\perp}$ times
\begin{equation}
\label{det-z11}
\det (z)_{11} = \mu^2|\xi|^2 b^{-1} \big(4 \sqrt{(1-t)(1-ut)} -(2-t)^2\big).
\end{equation}
Given $u\in]0,1/2[$, the unique zero $t\in]0,1[$
is found as the solution of Rayleigh's cubic equation,
\cite[(24)]{rayleigh:87:surface},
$0 = ((t-2)^4-16(1-t)(1-ut))/t$. 
Define the Rayleigh wave speed $c_r=c_s\sqrt{t}\in\Cinfty(X)$.
Set $p(\xi)=c_r|\xi|$.
The characteristic variety $\Sigma$ equals $\{p(\xi)=1\}$.
Thus \eqref{hypo-atmost} and \eqref{hypo-intersect} hold.
Obviously, $i(2b-t)\nu+t\sqrt{1-t}\,\xih \in\Kernel z(\xi)$, $\xi\in\Sigma$.
Observe that
\begin{equation}
\label{b-on-Sigma}
2(2b-t)=t(2-t)\quad\text{on $\Sigma$.}
\end{equation}
Thus
\begin{equation}
i(2-t)\nu+2\sqrt{1-t}\,\xih \in\Kernel z,\quad t=(c_r/c_s)^2,
\end{equation}
is a nowhere vanishing section of the kernel bundle.
Hence also \eqref{hypo-trivial} holds.
This example is of course well-known.
\end{example}

\begin{remark}
The identity \eqref{barnett-lothe-identity} goes back to
Barnett and Lothe; compare \cite[(3.18)]{lothe/barnett:76:exist-surface-wave}.
It is key to proving, in dimension three,
the uniqueness of subsonic traction-free surface waves
\cite[Theorem~8]{lothe/barnett:85:surface-wave-impedance-method}.
The second assumption in Proposition~\ref{prop-radial-line-and-charvar}
is needed to prove the existence of Rayleigh surface waves.
Compare with \cite[Theorem~12]{lothe/barnett:85:surface-wave-impedance-method},
where existence criteria are given in terms of the so-called
limiting velocity which corresponds to the boundary of the elliptic region.
See \cite[Theorem 2.2]{nakamura:91:rayleigh-pulses} for the
Barnett-Lothe condition in a microlocal setting, and the real principal
type property of the Lopatinski matrix it entails.
See \cite{Tanuma07Stroh} for a recent exposition of Barnett-Lothe theory,
and for a treatment of isotropic and transversely isotropic media.
\end{remark}

\section{Connections and Geometric Symbols}
\label{sect-conn-geosymb}

The elasticity operator is defined in terms of the Levi-Civita connection
and of the elasticity tensor.
We use the geometric \psdiff\ calculus of Appendix~\ref{app-geopsdo-calc}
to define and compute the leading symbol of the elasticity operator.
The leading symbol includes the principal and the subprincipal level.
The calculus depends on the choice of connections.

Equip $M$ with the Levi-Civita connection of $g$.
Let $\exp$ denote its exponential map.
If $x,y\in M$, then denote by $\tofrom{y}{x}$ the shortest geodesic segment from $x$ to $y$,
assuming its interior does not intersect the boundary, and that it is unique.

Let $E\ueber M$ be a (complex) vector bundle with connection $\nabla^E$.
Denote $\transport^E_{\gamma}\in\End(E_x,E_y)$ the parallel transport map
along a given curve $\gamma$ in $M$ from $x$ to $y$, e.g., $\transport^E_{\tofrom{y}{x}}$.
The connection can be recovered from its parallel transport maps:
\begin{equation}
\label{conn-from-transport}
\nabla^E_v s(x)=\frac{\intd}{\intd t}\Restrict{t=0} \transport^E_{\tofrom{x}{\exp_x tv}} s(\exp_x tv).
\end{equation}

Denote $\pi^* E\ueber T^* M$ the pullback of $E$
to the cotangent bundle $\pi:T^* M\to M$.
Let $a$ be a smooth section of $\pi^* E\ueber T^*M$.
Following \cite{Sharaf04geo1,Sharaf05geo2}, we introduce
the vertical and the horizontal covariant derivative of $a$.
The vertical derivative $\verderiv a(x,\xi)\in E_x\otimes T_xM$, at $\xi\in T^*_xM$,
is the derivative of the map $T^*_xM\to E_x$, $\xi\mapsto a(x,\xi)$.
The definition of the vertical derivative depends only on the linear
structure of the fibers of $T^*M$.
The horizontal derivative $\horderiv a(x,\xi)\in E_x\otimes T^*_xM$
is the derivative at $v=0$ of a map $T_x M\to E_x$,
\begin{equation}
\label{def-horderiv}
\horderiv a(x,\xi) = \frac{\partial}{\partial v}\Restrict{v=0}
   \transport^{E}_{\tofrom{x}{\exp_x v}} a(\exp_x v, \transport^{T^*M}_{\tofrom{\exp_x v}{x}}\xi).
\end{equation}
The horizontal derivative depends on the Riemannian structure and on the connection $\nabla^E$.
In the scalar case, $E=\C$, in local coordinates,
\[
\horderiv a(x,\xi) = \big(\partial_{x_j}a(x,\xi)
    +\Gamma^k_{ij}(x)\xi_k \partial_{\xi_i}a(x,\xi)\big)\intd x^j,
\]
where $\Gamma^k_{ij}$ denote the Christoffel symbols of the Levi-Civita connection.
Writing a local section of $\pi^* E$ as a sum of products $a_1(x,\xi)a_2(x)$
where $a_1$ is scalar and $a_2$ a section of $E$ one readily derives local formulas for
the horizontal derivative in terms of connection coefficients.
The vertical and the horizontal derivative extend to first order differential operators,
$\verderiv$ and $\horderiv$, which map sections of $\pi^* (E\otimes T^{r,s}M)$
to sections of $\pi^*(E \otimes T^{r+1,s}M)$ and of $\pi^*(E \otimes T^{r,s+1}M)$,
respectively.
The operators $\verderiv$ and $\horderiv$ commute.
It suffices to prove this when $E$ is the trivial line bundle, $E=\C$.
In this case the assertion is easily checked in normal coordinates.

Let $F\ueber M$ be another vector bundle.
Let $A:\Cinfty(M;E)\to\Cinfty(M;F)$ be a differential operator of order $m$.
We introduce a small parameter, $0<\h\leq 1$,
and we replace $A$ by the $\h$-differential operator $\h^mA$.
Then $A\in\Pskm{0}{m}(M;E,F)$ as a semiclassical (pseudo-)differential operator.
Refer to Appendix~\ref{app-geopsdo-calc} for an exposition of Sharafutdinov's
geometric \psdiff\ calculus in a semiclassical setting.
The formula~\eqref{testing-geom-symb} for the geometric symbol,
$\sigma_\h(A)\in\Sykm{0}{m}$, simplifies to
\begin{equation}
\label{def-geosymb-diffop}
\sigma_\h(A)(x,\xi)s = A_y\big(e^{i\langle\xi,\exp_x^{-1} y\rangle/\h}
               \transport^E_{\tofrom{y}{x}} s\big)\Restrict{y=x},
\end{equation}
where $\xi\in T^*_x M$, $s\in E_x$, and $i=\sqrt{-1}$.
The geometric symbol extends by continuity to the boundary of $M$.
In symbol computations we track the leading geometric symbol,
defined before Proposition~\ref{geocalc-adjoint}.
In the following, the symbol of an operator is always its geometric symbol.

For the Laplace-Beltrami operator one has $\sigma_\h(-\h^2\Delta)(x,\xi)=|\xi|^2$.
This is readily checked using normal coordinates.

From \eqref{def-geosymb-diffop} and \eqref{conn-from-transport} deduce
\begin{equation}
\label{geosymb-of-connection}
\sigma_\h(-i\h\nabla^E)(\xi)e=e\otimes \xi\in E_x\otimes T^*_xM.
\end{equation}
As before, to ease notation, we usually do not write the base point $x$
into the arguments of tensors and symbols.

If $E\ueber M$ is a Hermitian vector bundle then we define,
using the volume element $\dvol_M$, the Hilbert space $L^2(E)$.
Assume $E$ and $F$ are Hermitian vector bundles having metric connections.
The leading symbol of the formal adjoint $A^*$ of $A$ is given by
\begin{equation}
\label{maingeosymb-adjoint}
\sigma_\h(A^*) \equiv \sigma_\h(A)^* -i\h \trace\big(\verderiv\horderiv\sigma_\h(A)^*\big).
\end{equation}
See Proposition~\ref{geocalc-adjoint}.

Equip the bundle $E\otimes T^*M$ with the induced Hermitian structure
and the induced connection. The connection is metric.
Observe that the horizontal derivative of $\sigma_\h(-i\h\nabla^E)^*$ vanishes.
Therefore, \eqref{maingeosymb-adjoint} and \eqref{geosymb-of-connection} imply
\begin{equation}
\label{geosymb-of-adj-connection}
\sigma_\h\big((-i\h\nabla^E)^*\big)(\xi)(e\otimes\eta)= g(\xi,\eta) e,
\quad \xi,\eta\in T_x^* M, e\in E_x.
\end{equation}

By Proposition~\ref{geocalc-compos}
the leading symbol of a composition is given as follows:
\begin{equation}
\label{maingeosymb-compos}
\sigma_\h(AB) \equiv \sigma_\h(A)\sigma_\h(B) -i\h \trace\big(\verderiv\sigma_\h(A).\horderiv\sigma_\h(B)\big).
\end{equation}
The trace is the contraction of the $TM\otimes T^*M$ factor
which is produced by a pair of vertical and horizontal derivatives.
The dot terminates a differentiated expression, serving as a closing bracket.

Let $C\in\Cinfty(M;\End(E\otimes T^*M))$.
View $C$ as an operator which acts by multiplication on
sections of the bundle $E\otimes T^*M\ueber M$.
Let $\nabla$ denote the connection on the bundle $\End(E\otimes T^*M)\ueber M$
induced from the Levi-Civita connection and from $\nabla^E$.
Define sections $c, \divergence c$ of $\pi^*\End(E)\ueber T^*M$ as follows:
\begin{align*}
c(\xi)e&=\langle\xi, \pi^* C(e\otimes\xi)\rangle, \\
(\divergence c)(\xi)e&=\sum_j \langle\eta^j,(\pi^* \nabla_{v_j} C)(e\otimes\xi)\rangle,
\end{align*}
where the angular brackets denote contractions on covectors, using $g$.
Furthermore, $(v_j)$ and $(\eta^j)$ are any dual frames of $TM$ and $T^*M$.

\begin{lemma}
\label{symbol-of-nablaadj-C-nabla}
$\sigma_\h(-\h^2 {\nabla^E}^*\circ C\circ \nabla^E)=c - i\h \divergence c+\bigoh(\h^2)$.
\end{lemma}

\begin{proof}
Observe that $\sigma_\h(C)=\pi^* C$, and $\horderiv \pi^* C=\pi^* \nabla C$.
The symbol \eqref{geosymb-of-adj-connection} is linear in $\xi$.
Its vertical derivative is obvious. 
Using \eqref{maingeosymb-compos}, the symbol $a$ of $-i\h{\nabla^E}^*\circ C$ is found to be
\[
a(\xi)(e\otimes\eta) = \langle\xi, \pi^* C(e\otimes\eta)\rangle
      -i\h \sum_j \langle\eta^j  , \pi^* \nabla_{v_j} C (e\otimes\eta)\rangle.
\]
Here $(v_j)$ and $(\eta^j)$ are as in the definition of $\divergence c$.
The horizontal derivative of the symbol of $-i\h\nabla^E$ vanishes.
Therefore,
\[
\sigma_\h(-\h^2 {\nabla^E}^*\circ C\circ \nabla^E)(\xi)e
  = a(\xi) \sigma_\h(-i\h\nabla^E)(\xi)e = a(\xi)(e\otimes\xi),
\]
where we used \eqref{maingeosymb-compos}.
\end{proof}

Now assume $E=\CT M$ and $C$ the elasticity tensor.
Identify
\[
\End(\CT^{0,2}M)=\End(\CT M\otimes \CT^*M).
\]
Let $L$ the elasticity operator defined in \eqref{def-elast-traction}.
Recall from Riemannian geometry the following relation between the
Levi-Civita connection and the Lie derivative:
\begin{equation}
\label{geom-Liederiv-covderiv}
(\Lie_u g)(v,w)= g(\nabla_v u, w)+g(v,\nabla_w u),
\end{equation}
for (real) vector fields $u,v,w$.
Using the symmetries of the elasticity tensor we get
\[
L=\Def^* \circ C \circ \Def = (-i\nabla)^* \circ C \circ (-i\nabla).
\]
We obtain the following corollary to Lemma~\ref{symbol-of-nablaadj-C-nabla}:
\begin{equation}
\label{symb-L}
\sigma_\h(\h^2 L-\rho)=c -\rho\Id - i\h \divergence c+\bigoh(\h^2).
\end{equation}
If $C^{ijk\ell}$ respresent $C$ with respect to some given local coordinates,
then \eqref{symb-L} reads
\[
{\sigma_\h(\h^2 L-\rho)(\xi)}^{ik}
  = C^{ijk\ell}\xi_j \xi_\ell -\rho\delta^{ik}
     -\sqrt{-1} h C^{ijk\ell}_{\phantom{ijk\ell}|j} \xi_\ell
     +\bigoh(\h^2).
\]
The vertical bar followed by $j$ means covariant differentiation with respect to the $j$-th coordinate.
If the elastic medium is isotropic then the leading symbol becomes
\begin{equation}
\label{symb-L-iso}
\begin{aligned}
\sigma_\h(\h^2 L-\rho)(\xi)
     &\equiv \rho(c_p^2|\xi|^2 - 1) P(\xi) +\rho(c_s^2 |\xi|^2 -1)(\Id-P(\xi))\\
  &\phantom{==} -i\h \big(\nabla\lambda\otimes \xi +(\nabla\mu\otimes \xi)^*
                  +\langle\xi,\nabla\mu\rangle \Id\big),
\end{aligned}
\end{equation}
where $P(\xi)=\xih \otimes\xih$ denotes the orthogonal projection
to the propagation direction $\xih =\xi/|\xi|$.

\section{The Elasticity Operator in a Boundary Collar}
\label{sect-collar}

In a boundary collar, $]-\eps,0]\times X\subset M$,
we write the elasticity operator $L$ in terms differential operators on $X$
having coefficients which depend on $r\in I$, the negative distance to $X$.

Let $N(x)\in T_x M$ denote the unit exterior normal at $x\in X$.
There exists $\eps>0$ such that, if we set $I=]-\eps,0]$,
the exponential map of the Levi-Civita connection defines
a diffeomorphism onto a neighbourhood of $X$ in $M$:
\[
I\times X\to M,\quad (r,x)\mapsto y=\exp(rN(x)).
\]
Essentially without losing generality, we assume that this map is onto $M$.
The inverse map is $y\mapsto(r,x)$, where $-r=d(y,X)$ is the distance from $y$ to $X$,
and $x=p(y)$ is the unique point in $X$ closest to $y$.
The distance function $r$ satisfies the (eikonal) equation $|\nabla r|=1$ in $M$.
Extend $N$ to $M$ by $N=\nabla r$.
Also introduce the unit conormal field $\nu=\intd r$.
The level hypersurfaces
\[ M_r=\{y\in M\setof r+d(y,X)=0\} \]
are diffeomorphic to $X=M_0$.
The shape operator $S=\nabla N$ is a field of symmetric endomorphisms of $TM$, $g(Su,v)=g(u,Sv)$.
The second fundamental forms of the level hypersurfaces $M_r$
assign $(u,v)\mapsto -g(Su,v)$ (Weingarten equation).
The dependency of the metric tensor on $r$ is given by
the formula $(\Lie_{N} g)(v,w)=2g(Sv,w)$.
This formula follows from \eqref{geom-Liederiv-covderiv}.
Introduce $J\in\Cinfty(I_r\times X)$, the solution of
$\partial_r \log J= \trace S$, $J\restrict{r=0}=1$.
Then we have the following formula for the volume form of $M$:
\begin{equation}
\label{def-J}
\int_M f(y)\dvol_M(y) = \int_I\int_X f(\exp(rN(x)))J(r,x)\dvol_X(x)\intd r,
\end{equation}
$f\in\Ccinfty(M)$.
See \cite[Ch.~2]{petersen:98:riemannian-geometry} for
the geometry of hypersurfaces using distance functions.

Let $E\ueber M$ be a vector bundle with connection $\nabla^E$.
Denote $E_r\ueber M_r$ the bundles induced by the inclusions $M_r\subset M$, $r\in I$.
Set $E_X=E_0$.
Let $u\in\Cinfty(M;E)$ be a section of $E$.
Using parallel transport in $E$ along the geodesics which intersect the
boundary orthogonally, define $\widetilde{u}:I\to\Cinfty(X;E_X)$,
\[
\widetilde{u}(r)(x)=\widetilde{u}(r,x)= \transport^E_{\tofrom{x}{y}} u(y),
\quad\text{if $y=\exp(rN(x))$.}
\]
The map
\begin{equation}
\label{isom-Frechet-spaces}
\Cinfty(M;E)\to\Cinfty(I,\Cinfty(X;E_X)), \quad u\mapsto\widetilde{u},
\end{equation}
is an isomorphism of Fr\'echet spaces.
The isomorphism commutes with bundle operations
such as tensor products and contractions.

The covariant derivative in normal direction is transformed into $\partial_r$ under
the above isomorphism:
\begin{equation}
\label{normal-cov-derivative}
\widetilde{\nabla^E_N u}(r) =\partial_r \widetilde{u}(r),\quad r\in I.
\end{equation}
To see this,
consider the geodesic $I\to M$, $r\mapsto y(r)=\exp(rN(x))$.
The tangent vectors are $\dot{y}(r)=N(y(r))$.
Using \eqref{conn-from-transport}, it follows that
\begin{align*}
(\nabla^E_{N(y(r))} u)(y(r)) 
  &= \frac{d}{ds}\Restrict{s=r} \tau^E_{\tofrom{y(r)}{y(s)}} u(y(s)) \\
  &= \tau^E_{\tofrom{y(r)}{x}} \, \frac{d}{ds}\Restrict{s=r} \widetilde{u}(s,x).
\end{align*}
This implies \eqref{normal-cov-derivative}.
We have $\nabla_N N=SN=0$.
It follows that $\partial_r\widetilde{N}=0$, and $\partial_r\widetilde{\nu}=0$.
Abusing notation, we write $\partial_r$ to denote $\nabla^E_N$.

Define $\enu e=e\otimes \nu$ and $\inu (e\otimes \eta) = \langle \eta,\nu\rangle e$.
Notice that $\enu$ and $\inu$ commute with $\partial_r$.

Let $F\ueber M$ be a another vector bundle with a connection.
Let $B:\Cinfty(M;E)\to\Cinfty(M;F)$ be a differential operator.
Assume that $B$ is tangential.
This means, by definition, that $B$ commutes with the distance function $r$, $[B,r]=0$.
Then, for every $r\in I$, $B$ restricts to an operator
$B_r:\Cinfty(M_r;E_r)\to\Cinfty(M_r;F_r)$, $B_r U = (Bu)\restrict{M_r}$,
where $u$ is a section of $E\ueber M$ which extends a given section $U$ of $E_r\ueber M_r$.
The assumption $[B,r]=0$ implies that $B_r$ is well-defined.
Parallel transport along the geodesics orthogonal to $X$ defines
bundle isomorphisms $E_r\cong E_X$ and $F_r\cong F_X$.
Via these isomorphisms the $B_r$'s induce differential operators
$B(r):\Cinfty(X;E_X)\to\Cinfty(X;F_X)$, called associated with $B$, such that
$\widetilde{Bu}(r)= B(r)\widetilde{u}(r)$, $r\in I$.
Each $B(r)$ is a differential operator having coefficients which are $\Cinfty$
with respect to $r$.
Conversely, an operator $B$ is tangential if it is given in this way
by a family of differential operators $\{B(r)\setof r\in I\}$
with coefficients depending smoothly on $r$.

\begin{lemma}
\label{lemma-covderiv-tilde}
Let $E\ueber M$ be a real vector bundle with connection $\nabla^E$.
Then
\begin{equation}
\label{eq-covderiv-tilde}
\nabla^E = \enu \partial_r + B,
\end{equation}
where $B$ is tangential.
Moreover, $B(0) = \nabla^{E_X}$.
\end{lemma}
Here $E\otimes T^*M$ carries the induced connection.
The lemma extends, by decomposition into real and imaginary parts, to
complexifications of real bundles with connections.
In particular, it holds for complexified tensor bundles with the Levi-Civita connection.

\begin{proof}
Let $P^\perp,P^\parallel\in\Cinfty(M;\End(TM))$ denote the orthogonal projectors
onto the span of $N$ and onto its orthogonal complement, $N^\perp$, respectively.
Identify $E\otimes T^*M$ with $\Hom(TM,E)$.
Let $u\in\Cinfty(M;E)$.
We have the following decomposition in $\Cinfty(M;\Hom(TM,E))$:
\[
\nabla^E u= (\nabla^E u)P^\perp + (\nabla^E u)P^\parallel
= (\nabla^E_N u)\otimes\nu + Bu.
\]
This defines $B$, and implies \eqref{eq-covderiv-tilde}.
Note that $B$ is tangential.
We have
\[
B(0)\widetilde{u}(0) = Bu\restrict{X}
   =\big((\nabla^E u)P^\parallel\big)\restrict{X}
   =\big(\nabla^{E_X} (u\restrict{X})\big)\big(P^\parallel\restrict{X}\big).
\]
This proves the asserted formula for $B(0)$.
\end{proof}

Assume $E\ueber M$ a Hermitian bundle with a metric connection.
Using~\eqref{def-J}, and the fact that parallel transport preserves inner products,
we have
\begin{equation}
\label{Ltwo-prod-and-tilde}
\int_M (u|v)_E \dvol_M = \int_I \int_X (\widetilde{u}|\widetilde{v})_{E_X}J\dvol_X\intd r,
\end{equation}
if $u,v\in\Ccinfty(M;E)$.
Formal adjoints of differential operators on $M$ are taken with respect
to these inner products.
The inner product of sections $u$ and $v$ of $E_X\ueber X$ is
$\int_X (u|v)_{E_X} \dvol_X$.
Formal adjoints of operators $A(r)$ associated with a tangential operator
$A$ are defined with respect to this inner product.

Next we prove a formula which expresses the elasticity operator $L$
as a quadratic polynomial in $D_r=-i\partial_r$ with tangential coefficients.
Now assume $E=\CT M$, and let $B$ as in \eqref{eq-covderiv-tilde}.
Define tangential operators,
\begin{equation*}
A_0=\inu C\enu,\quad
A_1=-i\inu C B, \quad
A_2=B^*CB.
\end{equation*}
The order of $A_j$ is $j$.
Moreover, $A_1^*=iB^*C\enu$.
\begin{prop}
\label{prop-rep-elast-op}
The elasticity and traction operators defined in \eqref{def-elast-traction}
are as follows:
\begin{align*}
L&=(D_r -i \trace S)(A_0 D_r + A_1) + A_1^* D_r + A_2, \\
-iT&=A_0(0) D_r + A_1(0).
\end{align*}
Furthermore, $A_1^*(0)=A_1(0)^*$.
\end{prop}
\begin{proof}
Let $u,v\in\Ccinfty(M;\CT M)$.
It follows from \eqref{geom-Liederiv-covderiv}
and the symmetry properties of the elasticity tensor that
\[
\int_M \big(\Def u\mid\Def v\big)_C\dvol_M
  =\int_M \big(C\nabla u\mid\nabla v\big)\dvol_M.
\]
Inserting \eqref{eq-covderiv-tilde} and using the definition of $A_j$,
the right-hand side equals
\begin{align*}
\int_M & (\inu C\nabla u|\partial_r v)\dvol_M
     +\int_M (B^* C\nabla u|v)\dvol_M \\
   &= \int_M (A_0\partial_r u+iA_1 u|\partial_r v)\dvol_M
     +\int_M (-iA_1^* \partial_r u+ A_2 u|v)\dvol_M
\end{align*}
Partial integration with respect to $r$ is done,
using \eqref{Ltwo-prod-and-tilde}, as follows:
\begin{align*}
\int_M & (w|\partial_r v)\dvol_M
   = \int_I\int_X (\widetilde{w}|\partial_r \widetilde{v})J\dvol_X \intd r \\
  &= \int_X \big(w(0)|v(0)\big)\dvol_X
     - \int_I\int_X\big((\partial_r\log J)\widetilde{w}+\partial_r\widetilde{w}|\widetilde{v}\big)J\dvol_X\intd r.
\end{align*}
Summing up we have
\begin{align*}
\int_M & \big(\Def u \mid\Def v\big)_C\dvol_M \\
   &= \int_M \big((D_r-i\trace S)(A_0 D_r+A_1)u +A_1^*D_r u+A_2 u\mid v\big)\dvol_M \\
   &\phantom{==} + \int_X \big(A_0(0)(\partial_r u)(0) +iA_1(0)u(0)\mid v(0)\big)\dvol_X.
\end{align*}
Comparing with \eqref{def-elast-traction} the formulas for $L$ and $T$ follow.
The last assertion follows because $J=1$ at $X$.
\end{proof}

Next we compute the leading symbols of the operators (associated with) $A_j$.
The symbols are $r$-dependent sections of $\pi^* \End(\CT_X M)\ueber T^*X$.
Dropping tildes, the symbol of $A_0$ equals
\[
\sigma_\h(A_0)= a=c(\nu)\in\Cinfty(I,\Cinfty(T^*X;\pi^* \End(\CT_X M)).
\]
Introduce the divergence of the acoustic tensor restricted to $X$ as follows:
\[
(\divergence_X c)(\zeta)v=\sum\nolimits_\alpha \langle\eta^\alpha,(\pi^* \nabla_{v_\alpha} C)(v\otimes\zeta)\rangle,
\]
if $\zeta\in T_x^*M$, $v\in T_xM$, $x\in X$.
Here $(v_\alpha)$ and $(\eta^\alpha)$ are any dual frames of $TX$ and $T^*X$.
If local coordinates are chosen such that $r$ is one coordinate and the other
coordinates are constant along the geodesics orthogonal to $X$, then
$(\divergence_X c)(\zeta)^{ik} = C^{i\alpha k\ell}_{\phantom{i\alpha k\ell}|\alpha}\zeta_\ell$.
Here the summation convention is used with Latin indices refering to all coordinates,
and Greek refering to all coordinates except $r$.
We also need the contraction $\langle C, S\rangle\in\Cinfty(M;\End(TM))$
of the elasticity tensor with the shape operator, in coordinates,
\[
{\langle C, S\rangle}^{ik} = C^{ijk\ell} S_{j\ell}, \quad S_{j\ell} = \nu_{j|\ell}.
\]
(Because of $\nabla_N S=0$ one can also write Greek coordinates instead of $j$ and $\ell$.)
\begin{lemma}
\label{lemma-a12-minus}
Let $a_1$ and $a_2$ denote the principal symbols of the $\h$-differential
operators $\h A_1$ and $\h^2A_2$, respectively.
At $r=0$: $a_1(\xi)=c(\nu,\xi)$, and $a_2(\xi)=c(\xi)$.
On the leading symbol level,
$\sigma_\h(\h A_1)=a_1$,
$\sigma_\h(\h A_1^*)=a_1^*-i\h a_{1-}$,
and $\sigma_\h(\h^2A_2)=a_2-i\h a_{2-}+\bigoh(\h^2)$,
where, at $r=0$,
\[
a_{1-} =(\divergence_X c)(\nu) + \pi^*\langle C, S\rangle,
\quad
a_{2-}(\xi) =(\divergence_X c)(\xi).
\]
\end{lemma}
\begin{proof}
By Lemma~\ref{lemma-covderiv-tilde} we have
\[
\h A_1(0)= \inu C \circ(-i\h\nabla), \quad \h^2A_2(0)=(-i\h\nabla)^* \circ C \circ(-i\h\nabla),
\]
where $\nabla=\nabla^{TX}$ is the Levi-Civita connection of the boundary.
We compute the leading symbol of $\h A_1(0)$ using the composition formula \eqref{maingeosymb-compos}.
Recall \eqref{geosymb-of-connection}. 
The vertical derivative of the symbol of $\inu C$ vanishes,
Hence
\[
\sigma_\h(\h A_1)(0)(\xi)=\sigma_\h(\h A_1(0))(\xi)=c(\nu,\xi), \quad \xi\in T_X^* M.
\]
The formula for $\sigma_\h(\h^2A_2(0))$ follows from Lemma~\ref{symbol-of-nablaadj-C-nabla}.
In view of \eqref{maingeosymb-adjoint}, $a_{1-}=\trace \verderiv\horderiv a^*_1$.
Since $a^*_1(\xi)=c(\xi,\nu)=\langle\xi,\pi^*(C\enu)\rangle$ is linear in $\xi$,
its vertical derivative is immediate.
Hence
\[
\trace \verderiv\horderiv a^*_1 
 =\sum\nolimits_\alpha \langle\eta^\alpha,\pi^* \nabla_{v_\alpha} (C\enu)\rangle.
\]
Now, $\nabla_v (C\enu)$ equals $(\nabla_v C)\enu$ plus a contraction of $C$
with $\nabla_v\nu=Sv$, proving the formula for $a_{1-}$.
\end{proof}

If the elastic medium is isotropic, \eqref{C-iso-elast}, then a
straightforward computation shows that, at $r=0$,
\begin{align*}
(\divergence_X c)(\zeta) &= (\nabla\lambda\otimes\zeta)+ (\nabla\mu\otimes\zeta)^* + \langle\zeta,\nabla\mu\rangle\Id,\\
\langle C, S\rangle &= (\lambda+\mu) S + (\mu\trace S)\Id.
\end{align*}
Here $\nabla\lambda, \nabla\mu\in TX\subset T_XM$ are
the gradients of the Lam\'e parameters restricted to $X$.

\section{Microlocal Factorization}
\label{sect-factorization}

We factorize, microlocally in the elliptic region, the
$\h$-differential operator $\h^2L-\rho$ into a product
with right factor $\h D_r-Q$, where
$Q$ is a tangential $\h$-\psdiff\ operator such that the spectrum of its
principal symbol is contained in the lower halfplane, $\C_-$.

As in the previous section we identify $M$ with a boundary collar $I\times X$,
and sections of $\CT M\ueber M$ with $r$-dependent sections of $\CT_X M\ueber X$.
Operators are polynomials in $\hdr=\h D_r$ with tangential $\h$-(pseudo-)differential
operators as coefficients.
The latter are quantizations~\eqref{def-Oph-a},
$B_\h=\Op_\h(b_\h)\in\Pstkm{k}{m}$,
of tangential symbols,
\[
b_\h\in \Sytkm{k}{m}=\Cinfty(I,\Sykm{k}{m}(T^*X;\pi^*\End(\CT_X M))).
\]

By Proposition~\ref{prop-rep-elast-op}
the principal symbol $f(s,\xi)=c(\xi+s\nu)-\rho$ of $\h^2L-\rho$ at $\xi+s\nu$
is a second order polynomial in $s$.
View $s$ as the symbol of $\hdr$.
The coefficients are $\h$-independent tangential symbols.
By \eqref{C-is-elliptic}, there exists a constant $\delta>0$ such that
\begin{equation}
\label{f-est-from-below}
f(s,\xi)\geq \delta (1+|s|^2+|\xi|^2)\Id, \quad s\in\R,
\end{equation}
holds if $\xi$ is sufficiently large.
If $F\subset\EllRegion$ is closed and $R>0$, then $F\setminus\{|\xi|>R\}$ is compact.
Hence there exist $0<\varepsilon', \delta$ such that \eqref{f-est-from-below}
holds uniformly for $(r,\xi)\in [-\varepsilon',0]\times F$.
We say that a property holds at the elliptic region $\EllRegion$
if it is true in every open subset of $I\times\EllRegion$
where \eqref{f-est-from-below} holds uniformly.

Recall from section~\ref{section-impedance} that
we have a unique spectral factorization \eqref{factor-with-q} at $\EllRegion$.

\begin{lemma}
\label{factor-symbol}
Let $q=q(\xi)$, $\xi\in\EllRegion$, the unique solution of the spectral factorization
$f(s,\xi)=(s-q(\xi)^*)a(s-q(\xi))$, $\spec q(\xi)\subset\C_-$.
Then $q\in\Sytm{1}$ at $\EllRegion$.
\end{lemma}

\begin{proof}
By Lemma~\ref{factor-sapolyn} we have $a\,q = -\pi i f_0^{-1} + f_1 f_0^{-1}$
with integrals $f_j=f_j(\xi)$ defined there.
Using \eqref{f-est-from-below}, we can estimate $f_0(\xi)=\int f(s,\xi)^{-1}\intd s$ as follows:
\[
|f_0(\xi)|\leq\int_{-\infty}^\infty \delta^{-1} (1+|s|^2+|\xi|^2)^{-1}\intd s
      =\pi/\delta \jap{\xi},
\]
$\jap{\xi}=(1+|\xi|^2)^{1/2}$.
The integrand $f(s,\xi)^{-1}$ remains integrable after applying $\partial_r$,
$\horderiv$, and $\verderiv$ finitely many times.
Therefore these derivatives can be interchanged with the integral.
In view of the symbol properties of $f$, we deduce, using estimates as above,
$f_0\in\Sytm{-1}$ at $\EllRegion$.
Using an upper bound $f(s,\xi)\leq \delta^{-1} (|s|^2+\jap{\xi}^2)\Id$, we derive
$f_0(\xi)\geq \delta\jap{\xi}^{-1}\Id$, again in the sense of selfadjoint maps.
Therefore $f_0$ is an elliptic symbol, and $f_0^{-1}\in\Sytm{1}$ at $\EllRegion$.

Write $f_1=f_{10}+f_{11}$, where $f_{10}(\xi)= \int_{|s|\leq 1} s a f(s,\xi)^{-1}\intd s$,
\[
f_{11}(\xi) = \int_{|s|> 1} s^{-1} \big(s^2 a -f(s,\xi)\big) f(s,\xi)^{-1}\intd s.
\]
Recall $s^2 a -f(s)= -s(a_1+a_1^*)-(a_2-\rho)$.
Reasoning as in the proof of $f_0\in\Sytm{-1}$, we see that
the integrand of $f_{11}$ and its derivatives are integrable.
Moreover, we deduce $f_{11}\in\Sytm{0}$.
It is easy to see that $f_{10}\in\Sytm{-2}$.
Therefore, at $\EllRegion$, $f_1\in\Sytm{0}$.
The lemma follows.
\end{proof}

For a $\h$-tempered family $(u_\h)\in \h^{-\infty}\Dprime(X)$ the semiclassical wavefront set
$\WF_\h(u_\h)\subset T^*X\sqcup S^*X$ is defined, \cite{Gerard88asymptPoles},
\cite{SjoesZworski02Monodromy}.
Below we deal with operators associated to symbols which are not defined
on all of $T^*X$ but only at $\EllRegion$.
These operators are defined microlocally in $\EllRegion$ by letting them
operate on the subspace of distributions $(u_\h)$ which satisfy $\WF_\h(u_\h)\subset\EllRegion$,
modulo the space $\h^\infty\Cinfty$.

\begin{lemma}
\label{lemma-factorization}
Let $q$ be as in Lemma~\ref{factor-symbol}.
Microlocally at $\EllRegion$, 
\begin{equation}
\label{eq-factorization}
\h^2L-\rho = (\hdr - Q^\sharp)A_0(\hdr-Q),
\end{equation}
where $Q,Q^\sharp\in\Pstkm{0}{1}$, such that $Q-\Op_\h(q), Q^\sharp-\Op_\h(q^*)\in\Pstkm{-1}{0}$.
Here $A_0$ is as in Proposition~\ref{prop-rep-elast-op}.
\end{lemma}

\begin{proof}
Initially we set $Q=\Op_\h(q)$ and $Q^\sharp=\Op_\h(q^*)$.
At $\EllRegion$,
\begin{equation}
\label{factor-with-error}
\h^2L-\rho = (\hdr - Q^\sharp)A_0(\hdr-Q) + R_1 + R_0\,\hdr,
\end{equation}
where $R_j\in\Pstkm{-1}{j}$.
Here we used the formula for $L$ given in Proposition~\ref{prop-rep-elast-op}.
Observe that, if $A\in\Pstkm{k}{m}$, then the commutator $[\hdr,A]$
belongs to $\Pstkm{k-1}{m}$.
Aiming at an inductive construction, we assume that \eqref{factor-with-error} holds
for some positive integer $k$ with $R_j\in\Pstkm{-k}{j+1-k}$.
The spectra of $q$ and $q^*$ are disjoint.
It follows that the equation $s q- q^* s=r$ has, at $\EllRegion$, 
for every symbol $r\in \Sym{m}$ a unique solution $s\in \Sym{m-1}$.
Applying this construction to the principal symbols of the $R_j$'s,
we find operators $S_j\in\Pstkm{-k}{j-k}$ such that
$S_jQ -Q^\sharp S_j-R_j\in\Pstkm{-k-1}{j-k}$.
Set
\[
Q_1 =Q - A_0^{-1}(S_0Q+S_1),\quad
Q_1^{\sharp}=Q^\sharp + (Q^\sharp S_0+S_1) A_0^{-1}.
\]
Then
\begin{align*}
(\hdr - Q_1^{\sharp}) & A_0(\hdr-Q_1) \\
  &= (\hdr - Q^\sharp)A_0(\hdr-Q) \\
  &\phantom{=} + \big(S_0Q -Q^\sharp S_0\big)\hdr + \big(S_1Q -Q^\sharp S_1\big) \\
  &\phantom{=} + [\hdr, S_0 Q+S_1]  -(Q^\sharp S_0+ S_1) A_0^{-1} (S_0Q+S_1).
\end{align*}
Replace $Q$ and $Q^\sharp$ by $Q_1$ and $Q_1^{\sharp}$, respectively.
Then, by the symbol calculus, \eqref{factor-with-error} holds
with smaller errors, $R_j\in\Pstkm{-k-1}{j-k}$.
The proof is completed using asymptotic summation.
\end{proof}
It follows from the foregoing construction that the symbol of $Q$ is classical.

\section{A Dirichlet Parametrix}
\label{sect-parametrix}

Microlocally at $\EllRegion$, we solve, constructing a parametrix,
$B f=u$, the Dirichlet problem $\h^2Lu-\rho u=0$, $u\restrict{X}=f$.
We adapt the method of \cite[7.12]{taylor:96:pde2} to our setting.

Denote $\Sypm{m}\subset\Cinfty([-1,0],\Cinfty(T^* X;\pi^*\End(\CT_X M)))$
the space of symbols $b(s,\eta)$, $-1\leq s\leq 0$, $\eta\in T^*X$,
which satisfy the estimates
\[
 |\partial_s^\tau(\verderiv)^j(\horderiv)^\ell b(s,\eta)\big|
    \leq C_{\tau j\ell} \jap{\eta}^{m+\tau-j},
\]
for all nonnegative integers $\tau$, $j$, and $\ell$.
Let $\Sypkm{k}{m}$ denote the corresponding space of $\h$-dependent symbols $b_\h$.
Observe that $g(s\jap{\eta})\in\Sypm{0}$ if $g(t)=|t|^je^{\varepsilon t}$, $\varepsilon>0$,
$j$ a nonnegative integer.

We continue to work in a collar $I\times X\subset M$.
Choose a cutoff $\chi_0$ as in \eqref{def-Oph-a}.
Let $\delta>0$.
Given $b_\h\in\Sypkm{k}{m}$ introduce the operator
$B_{\h}=\Op_{\delta,\h}(b_\h(r/\h))$
as follows:
\begin{equation}
\label{Dirichlet-Param}
\begin{aligned}
B_{\h} f(r,y)=
   (2\pi\h)^{-n} & \int \limits _{T^*_y} \int \limits _{T_y}
      e^{-i\langle\eta,v\rangle /\h +\delta r\jap{\eta}/\h} \chi_0(y,v) \\
    & \cdot b_\h(r/\h,y,\eta) 
       \transport^{\CT_X M}_{\tofrom{y}{\exp_y v}} f(\exp_y v)\intd v \intd\eta,
\end{aligned}
\end{equation}
$r\in I$, $n=\dim X$.
We call $B_{\h}$ a Poisson operator with symbol $b_\h$ and (exponential) decay $\delta$.
The arguments in \cite[Ch.~7 Prop.~12.4]{taylor:96:pde2} apply to give
$B_\h:L^2(X)\to H^{-m+1/2}_\h(I\times X)$ with norm $\bigoh(\h^{-k+1/2})$.
(The Sobolev spaces $H^s_\h$ are defined using $\h D$ instead of $D$.)
If $0<\delta'<\delta$, $j\in\N$,
then $r^j B_\h\in\Op_{\delta',\h}\Sypkm{k-j}{m-j}$.
Moreover, $B_\h f\in\Cinfty$ in $r<0$,
and $B_\h f(r)$ decays together with its derivatives as $e^{\delta' r/\h}$,
uniformly if $f$ ranges in a bounded subset of $L^2(X)$.
We call $\h$-dependent operators negligible if they have Schwartz kernels
which are smooth and $\bigoh_{\Cinfty(M\times X)}(\h^\infty)$.
We write $A\equiv B$ iff $A-B$ is negligible.
Note that $B_\h$ in \eqref{Dirichlet-Param} is negligible if there exists
$\epsilon>0$ such that $b_\h(s,\eta)=0$ if $-\epsilon<s\leq 0$.

We need to handle the composition of a Poisson operator with a tangential operator.
The following lemma deals with this when the symbols are classical, i.e.,
they possess asymptotic expansions in powers of $\h$.

\begin{lemma}
\label{lemma-tang-compos-poisson}
Let $0<\delta'<\delta$.
Let $A=\Op_\h a(r)$ and $B=\Op_{\delta,\h}b(r/\h)$,
where $a=a(r,\eta)\in\Sytm{1}$ and $b(s,\eta)\in\Sypm{m}$
are $\h$-independent symbols.
Then $AB\equiv\Op_{\delta',\h}c(r/\h)$, where $c=c_\h\in\Sypkm{0}{m+1}$
has an asymptotic expansion $c\sim\sum_{j\geq 0} \h^j c_j$, $c_j\in\Sypm{m+1-j}$.
The principal term equals
\[ c_0(s,\eta)= a(0,\eta)b(s,\eta) e^{(\delta-\delta')s\jap{\eta}}. \]
\end{lemma}
\begin{proof}
Using Taylor expansions,
$a(r,\eta)=\sum_{j<N} r^j a_j(\eta) + r^N a_N'(r,\eta)$,
and the properties of $r^j B$ noted above,
we may assume without loss of generality that $a$ does not depend on $r$.
Arguing as in the proof of Proposition~\ref{geocalc-compos} we
can write, at least formally, $AB=\Op_{0,\h}\tilde{c}(r/\h)$, where
\begin{align*}
\tilde{c}(s,x,\xi)
  = (2\pi\h)^{-2n} \int \limits _{T_x\times T^*_x\times T_x\times T^*_x}
    & e^{i\varphi/\h} a(x,\eta) \transport^{\pi^*\End(\CT_X M)}_{\tofrom{x}{y}} b(s,y,\zeta)
       e^{\delta s\jap{\zeta}} \\
    & \cdot  M(x,w+v,v) \intd(v,\eta,w,\vartheta),
\end{align*}
$\varphi$ as in \eqref{c-statphase}.
We use the standard arguments in handling compositions of symbols:
dyadic decompositions and the method of (non-)stationary phase.
We infer that there exist $\epsilon>0$
and $d_j\in\Sypm{m+1-j}$, $d_0(s,\eta)=a(\eta)b(s,\eta)$,
such that for every $N$,
\[
\tilde{c}(s,\eta) =\big(\sum_{j<3N} \h^j d_j(s,\eta)\big)e^{\delta s\jap{\eta}}
    + \tilde{d}_{N\h}(s,\eta)e^{\epsilon s\jap{\eta}},
\]
where $\tilde{d}_{N\h}\in \Sypkm{-N}{m+1-N}$.
Observe that $\jap{\xi}/\jap{\eta}$ is uniformly bounded from below
if $\xi$ and $\eta$ range in the same dyadic shell.
Above we have chosen $\epsilon$ less than $\delta$ times this bound.
Define $c_\h(s,\eta)$ as the product of an asymptotic sum $\sum_{j\geq 0} \h^j d_j(s,\eta)$
with the symbol $e^{(\delta-\delta')s\jap{\eta}}\in\Sypm{0}$.
It follows that $AB-\Op_{\delta',\h}c(r/\h)$ belongs to
$\Op_{\epsilon,\h}\Sypkm{-N}{m+1-N}$ for every $N$.
Thus $AB\equiv\Op_{\delta',\h}c(r/\h)$.
\end{proof}

Let $q$ and $Q$ as in Lemma~\ref{lemma-factorization}.
If $\eta\in\EllRegion$ ranges in a set having a positive distance to the
complement of the elliptic region, then there exist
positive constants $\delta_0$ and $M$ such that
\begin{equation}
\label{bound-exp-q}
|e^{si q(0,\eta)}|\leq M e^{s\delta_0\jap{\eta}}, \quad s\leq 0.
\end{equation}
This follows from the fact that the spectrum of $q(0,\eta)/\jap{\eta}$
is contained in a compact subset of the lower halfplane then.
We shall solve $(\hdr-Q)B\equiv 0$, $B\restrict{r=0}=\Id$,
microlocally at $\EllRegion$.
On the symbol level we have to solve linear ordinary differential equations
with constant coefficient matrices.
The following assertions are true microlocally in $\EllRegion$
where \eqref{bound-exp-q} holds.

\begin{lemma}
\label{lemma-ode-for-symbols}
Let $0<\delta<\delta_0$.
Let $r\in \Sypm{1+m}$ and $v\in\Sym{m}$.
Let $b(s,\eta)$ be the solution of the initial value problem
\begin{equation}
\label{ode-for-ph}
\partial_s b(s,\eta)=\big(i q(0,\eta) -\delta\jap{\eta}\big) b(s,\eta)+r(s,\eta), \quad -1<s\leq 0,
\end{equation}
and $b(0,\eta)=v(\eta)$.
Then $b(s,\eta)\in\Sypm{m}$.
\end{lemma}

\begin{proof}
Note that the coefficient matrix of \eqref{ode-for-ph} does not depend on $s$.
Representing $b$ by Duhamel's formula and using \eqref{bound-exp-q} 
we derive the estimate
\begin{align*}
|b(s,\eta)| &\leq M|v(\eta)|+M\int_s^0 e^{(\delta_0-\delta) s\jap{\eta}} |r(s,\eta)|\intd s \\
            &\leq M|v(\eta)|+(M/(\delta_0-\delta)) \sup_{s\leq 0} |r(s,\eta)|/\jap{\eta}.
\end{align*}
Moreover, we can estimate $\partial_s b(s,\eta)$ by estimating the right-hand side of \eqref{ode-for-ph}.
Differentiating \eqref{ode-for-ph} we derive linear ordinary differential equations
for $\partial_s^\tau(\verderiv)^j(\horderiv)^\ell b(s,\eta)$.
These equations are of the same structure as \eqref{ode-for-ph} with
the same coefficient matrix.
The asserted symbol estimates are obtained recursively.
\end{proof}

\begin{prop}
\label{exist-dirichlet-param}
Let $0<\delta<\delta_0$, and $\epsilon>0$.
There exists $B\in\Op_{\delta,\h}\Sypkm{0}{0}$ with Schwartz kernel
supported in $-\epsilon<s\leq 0$, such that, microlocally at $\EllRegion$,
$(\hdr -Q)B\equiv 0$ and $B\restrict{r=0}=\Id$.
Moreover, $(\h^2 L-\rho)B\equiv 0$.
\end{prop}

\begin{proof}
It follows from Lemma~\ref{lemma-tang-compos-poisson} that,
for a classical symbol $b\in\Sypkm{k}{m}$, $0<\delta'<\delta$,
modulo negligible operators,
the composition $(\hdr-Q)\Op_{\delta,\h}b(r/\h)$ equals
$\Op_{\delta',\h}c(r/\h)$, $c\in\Sypkm{k}{m+1}$.
Moreover, $c$ is classical, and, modulo $\Sypkm{k-1}{m}$,
\[
c(s,\eta)\equiv
\big(-i\partial_s b(s,\eta)-i\delta\jap{\eta} b(s,\eta) -q(0,\eta) b(s,\eta)\big)e^{(\delta-\delta')s\jap{\eta}}.
\]
Fix a sequence $(\delta_j)$, $\delta<\delta_{j+1}<\delta_j$.
Using Lemmas~\ref{lemma-tang-compos-poisson} and \ref{lemma-ode-for-symbols}
we recursively find $\h$-independent symbols $b_j\in\Sypm{1-j}$, $b_1\restrict{r=0}=\Id$,
$b_j\restrict{r=0}=0$ if $j>1$, such that $B_j=\h^{j-1}\Op_{\delta_j,\h}b_j(r/\h)$ satisfy
\[
(\hdr-Q)(B_1+\dots+B_j)\in\Op_{\delta',\h}\Sypkm{-j}{1-j},\quad \delta_{j+1}<\delta'<\delta_j.
\]
Now $B$ is contructed using asymptotic summation.
The last assertion follows from the factorization \eqref{eq-factorization}.
\end{proof}

\section{The Displacement-to-Traction Operator}
\label{sect-dt-op}

In this section we deal with operators on the boundary $X$.
Therefore, in the following, operators and symbols are, as a rule, evaluated at $r=0$.

Let $B$ denote the Dirichlet parametrix given in Proposition~\ref{exist-dirichlet-param}
and $T$ the traction defined in \eqref{def-elast-traction}.
The operator $Z=\h TB$ is called the semiclassical displacement-to-traction operator,
or Neumann operator, at $\EllRegion$.
By Propositions~\ref{prop-rep-elast-op} and \ref{exist-dirichlet-param} we have,
if $\WF_\h(f)\subset\EllRegion$,
\[
Zf=(iA_0\hdr B f + i\h A_1 B f)\restrict{X} = iA_0(0) Q(0)f +i\h A_1(0)f.
\]
Therefore, $Z=iA_0Q+i\h A_1$, and $Z$ is, microlocally in $\EllRegion$,
a \psdiff\ operator of class $\Pskm{0}{1}$.
The symbol of $Z$ is classical since the symbols of $A_j$ and $Q$ are.

\begin{lemma}
\label{Z-op-mainsymb}
The displacement-to-traction operator $Z$ is, in $\EllRegion$,
up to a negligible operator, formally selfadjoint.
The principal symbol of $Z$ equals the surface impedance tensor
\begin{equation}
\label{z-of-Z}
z=i(aq+a_1)\in\Sym{1}(\EllRegion;\pi^*\End(\CT_X M)).
\end{equation}
The leading symbol of $Z$ is $z+\h z_-$, where $z_-\in\Sym{0}$,
\begin{equation}
\label{sub-z-of-Z}
z_- q -q^*z_-  = i\trace(S) z +i\partial_r z -a_{2-} -a_{1-}q 
  +\trace\big(\verderiv q^*. a\horderiv q\big).
\end{equation}
\end{lemma}
\begin{proof}
Let $f_1, f_2\in L^2(X;\CT_X M)$, $\WF_\h(f_j)\subset\EllRegion$,
and set $u_j =B f_j$.
By \eqref{def-elast-traction},
\begin{align*}
\int_X & (Zf_1|f_2)\dvol_X - \int_X (f_1|Zf_2)\dvol_X  \\
   &= \h^{-1} \int_M (u_1|\h^2 Lu_2-\rho u_2) -(\h^2 Lu_1-\rho u_1|u_2) \dvol_M.
\end{align*}
It follows from Proposition~\ref{exist-dirichlet-param} that the right-hand side
is $\bigoh(\h^\infty)$, uniformly if the $f_j$'s range in a bounded set
and have $\h$-wavefronts contained in a common closed subset of $\EllRegion$.
Thus $Z^*=Z$ in $\EllRegion$.
Recalling $Z=iA_0Q+i\h A_1$, we infer from the symbol calculus
that $z=i(aq+a_1)$ is the principal symbol.

It remains to prove the formula for $z_-$.
Write the leading symbols of $Q$ and $Q^\sharp$ as $q+\h q_-$ and $q^*+\h q_-^\sharp$, respectively.
It is easy to see that $z_-=iaq_-$.
Recall the formula for $L$ in Proposition~\ref{prop-rep-elast-op}.
The factorization~\eqref{eq-factorization} is equivalent to
\begin{align*}
(\hdr & -i\h\trace(S)) \h A_1 +(\h A_1^* -i\h\trace(S) A_0)\hdr +\h^2A_2 -\rho \\
   &= -\hdr A_0 Q -Q^\sharp A_0\hdr + Q^\sharp A_0 Q.
\end{align*}
This in turn is equivalent to the following two equations of
tangential operators:
\begin{gather*}
\h A_1 +\h A_1^* -i\h\trace(S) A_0 + A_0 Q + Q^\sharp A_0 = 0, \\
[\hdr, A_0Q+ \h A_1] -i\h\trace(S) \h A_1 +\h^2A_2 -\rho - Q^\sharp A_0 Q =0.
\end{gather*}
On the principal symbol level these equations become
$a_1+a_1^* +aq+q^*a=0$ and $a_2-\rho-q^* a q=0$.
These equations agree with \eqref{factor-with-q}.
On the leading symbol level the equations become,
after division by $\h$,
\begin{gather*}
-ia_{1-} -i\trace(S) a + aq_- + q_-^\sharp a -i\trace\big(\verderiv q^*. \horderiv a\big) =0, \\
-\partial_r z -i\trace(S) a_1 -ia_{2-} 
   - q^* a q_- -q_-^\sharp aq +i\trace\big(\verderiv q^*. \horderiv aq\big) =0.
\end{gather*}
Elimination of $q_-^\sharp$ from these equations gives
\[
(aq_-)q -q^*(aq_-) = ia_{1-}q +\trace(S) z +\partial_r z +ia_{2-}
   -i\trace\big(\verderiv q^*. a\horderiv q\big).
\]
Formula~\eqref{sub-z-of-Z} for $z_-=iaq_-$ follows.
\end{proof}

In principle $z_-$ is found as the unique solution of the linear equation \eqref{sub-z-of-Z}.
The right-hand side of the equation consists of known quantities
and their first order derivatives.
Refer to section~\ref{sect-isosubpr} for an algorithm computing $z_-$
if the elastic medium is isotropic.

\section{Diagonalization of $Z$}
\label{sect-diag}

Assume \eqref{hypo-atmost} and \eqref{hypo-intersect}.
By Proposition~\ref{prop-radial-line-and-charvar} the kernel $\Kernel z$ defines
a line bundle over the characteristic variety $\Sigma=p^{-1}(1)$ of the
surface impedance tensor $z$.
Since zero is a simple eigenvalue of $z$ at $\Sigma$,
there exist $\epsilon>0$ and an open neighbourhood $K\subset\EllRegion$ of $\Sigma$ such that
$z(\xi)$, $\xi\in K$, has exactly one eigenvalue $\lambda_0(\xi)$ of modulus $<\epsilon$.
(In the following, $K$ is to be replaced by a smaller neighbourhood when necessary.)
The line bundle $E_0=\Kernel(z-\lambda_0)\ueber K$ is a subbundle
of $\pi^*\CT_X M=\Hom(\C,\pi^*\CT_X M)$.
The orthoprojector onto this bundle is given by a contour integral,
$u_0=(2\pi i)^{-1}\oint_{|\lambda|=\epsilon} (\lambda-z)^{-1}\intd\lambda$.
Denote $u_1=\Id-u_0$ the orthoprojector onto the orthogonal bundle, $E_1$.

Assume also \eqref{hypo-trivial}.
Choose a unit section $v$ of $\Kernel z\ueber \Sigma$, $|v|=1$.
Using $u_0$, extend $v$ to a unit section of $E_0\ueber K$.
Call this section also $v$.
Clearly, $u_0=v\otimes v^*$.
If $R\in\Pskm{0}{0}$ denotes the inverse
of a square root of the scalar operator ${\Op_\h(v)}^* \Op_\h(v)$,
then $V=\Op_\h(v)R$ satisfies $V^*V=\Id$, i.e., $V$ is an isometry.

\begin{lemma}
\label{lemma-diagonalize}
Choose $V\in\Pskm{0}{0}(K;\C,\CT_X M)$,
with principal symbol $v$, such that $V^*V=\Id$.
Set $U_0=VV^*$, $U_1=\Id-U_0$.
There exist
$B\in\Pskm{-2}{-1}(X;\CT_X M)$, $B^*=B$, and
$R\in\Pskm{-1}{-1}(X;\CT_X M)$, $R^*+R=0$,
such that, microlocally in $K$,
\begin{equation}
\label{Z-decomp}
(\Id-R^*)Z(\Id-R)= U_0(Z+B)U_0 + U_1(Z+B)U_1.
\end{equation}
In particular,
\begin{equation}
\label{Z-itwine}
(\Id-R^*)Z(\Id-R)V= VV^*(Z+B)V.
\end{equation}
The leading symbol of the scalar operator $V^*(Z+B)V\in\Pskm{0}{1}$ equals
\begin{equation}
\label{Z-itwine-mainsymb}
\lambda_0 +\h(z_- v|v) 
      -i\h\trace\big(v^*\verderiv z. \horderiv v +\verderiv v^*.\horderiv \lambda_0.v\big).
\end{equation}
Here, as in Lemma~\ref{Z-op-mainsymb}, $z+\h z_-$ denotes the leading symbol of $Z$.
\end{lemma}
\begin{proof}
To prove \eqref{Z-decomp} we adopt ideas of \cite{Stefanov00LowerBd}.
The operators $U_0$ and $U_1$ are orthogonal projectors,
$U_j^*=U_j=U_j^2$, and $U_1U_0=0$.
Write $Z=U_0ZU_0+U_1ZU_1+B$, where $B=U_0ZU_1+U_1ZU_0$.
Since $u_jz=zu_j$ and $u_1u_0=0$ we have $B\in\Pskm{-1}{0}$.
Let $\h b$, $b=b^*\in\Sym{0}$, denote the principal symbol of $B$.
Define the section $z_j=z\restrict{E_j}$ of $\End(E_j)$.
The spectra of $z_0$ and $z_1$ are disjoint.
Therefore the Sylvester equation $sz_0-z_1s=u_1bu_0$ has a
unique solution $s$ which is a section of $\Hom(E_0,E_1)$.
We extend $s$ to a section of $\End(\pi^*\CT_X M)$ by $s=u_1su_0$.
Then $sz-zs=u_1bu_0$, and $s\in\Sym{-1}$.
Define $S=\Op_\h(\h s)$ and $R=U_0S^*U_1-U_1SU_0$.
Then, $R^*=-R$ and
$B=U_0BU_1+U_1BU_0 \equiv R^*Z+ZR$ modulo $\Pskm{-2}{-1}$.
Therefore, with a different $B\in\Pskm{-N}{-N+1}$, $N=2$, and $Z_0=Z_1=Z$, we have
\begin{equation}
\label{Z-decomp-step}
(\Id-R^*)Z(\Id-R)= U_0Z_0U_0 + U_1Z_1U_1 +B.
\end{equation}
If $N\geq 2$ then, using the same construction as before,
we find $R_1\in\Pskm{-N}{-N}$, $R_1^*=-R_1$ such that
$U_0BU_1+U_1BU_0 \equiv R_1^*Z+ZR_1$ modulo $\Pskm{-N-1}{-N}$.
Hence we get \eqref{Z-decomp-step} with $R$ and $Z_j$ replaced by
$R+R_1$ and $Z_j+B$, respectively.
The new error $B$ belongs to $\Pskm{-N-1}{-N}$.
Iterating this construction and using asymptotic summation \eqref{Z-decomp} follows.
Since $U_0V=V$, \eqref{Z-decomp} implies \eqref{Z-itwine}.

Observe that the leading symbols of $V^*(Z+B)V$ and $V^*ZV$ are equal.
The principal symbol equals $(v|zv)=\lambda_0$ because $|v|=1$.
We write the leading symbol of $V$ as $(1+\h\gamma)v+ \h w$, where $v^*w=(w|v)=0$.
Note $(v|zw)=0$.
A straightforward symbol computation, using
\eqref{geosymb-adjoint} and \eqref{geo-symb-compos}, gives
\begin{align*}
\sigma_\h(V^*ZV) &\equiv \lambda_0 +\h(z_- v|v) +\h(\gamma+\overline{\gamma})\lambda_0 \\
   &\phantom{==} -i\h\trace\big(v^*\verderiv z. \horderiv v +\verderiv\horderiv v^*.zv +\verderiv v^*.\horderiv zv\big)
\end{align*}
modulo $\bigoh(\h^2)$.
From $V^*V=1$ it follows that the leading symbol of $V^*V$ equals unity.
Since $|v|^2=1$ is the principal symbol, this implies
\[
\h(\gamma+\overline{\gamma}) -i\h\trace \big(\verderiv\horderiv v^*.v + \verderiv v^*. \horderiv v\big)=0.
\]
Therefore the expression for the symbol of $V^*ZV$ simplifies to
\begin{align*}
\sigma_\h(V^*ZV) &\equiv \lambda_0 +\h(z_- v|v) \\
   &\phantom{==} +i\h\trace\big(\lambda_0\verderiv v^*. \horderiv v
         - v^*\verderiv z. \horderiv v -\verderiv v^*.\horderiv zv\big)
\end{align*}
modulo $\bigoh(\h^2)$.
Using $\horderiv zv= \lambda_0\horderiv v+\horderiv\lambda_0.v$
we deduce \eqref{Z-itwine-mainsymb}.
\end{proof}

Denote $\Psmphg{m}$ the class of $\h$-independent \psdiff\ operators $A$
with polyhomogeneous symbols, $a\sim\sum_{j\leq m} a_j$,
$a_j$ homogeneous of degree $j$.
When regarded as an $\h$-dependent operator, $A\in\Pskm{m}{m}$ has 
the classical symbol $\sum_{j\leq m} \h^{-j} a_j$.
In the next lemma, following 
\cite{PopovVodev99ResoTransm} and \cite{Stefanov00LowerBd},
we use this relation to conjugate the scalar operator
constructed in Lemma~\ref{lemma-diagonalize} into $\h P-1$,
where $P$ is $\h$-independent.

Recall that $\halfdens\ueber X$ denotes the bundle of half-densities.

\begin{lemma}
\label{lemma-P}
There is a selfadjoint operator $P\in\Psmphg{1}(X;\halfdens)$ with principal symbol $p$,
and an operator $A\in\Pskm{0}{0}$ from half-density sections to scalar functions,
elliptic near $\Sigma$,
such that $A^*V^*(Z+B)VA=\h P-1$ in a neighbourhood of $\Sigma$.
The subprincipal symbol of $P$ equals, on $\Sigma$,
\begin{equation}
\label{subprinc-P}
\begin{aligned}
\subprinc{p} &= {(\dot{z}v|v)}^{-1}\big(\RE (z_- v|v) +\IM\trace(v^*\verderiv z.\horderiv v)\big)  \\
  &\phantom{==} + \IM \trace\horderiv p.\verderiv v^*.v.
\end{aligned}
\end{equation}
Here $\dot{z}$ denotes the radial derivative of $z$.
If instead of $v$ another unit section $\tilde{v}=e^{i\varphi} v$ of $\Kernel z\ueber \Sigma$
is used to define $V$, and thus $P$, then the principal symbol of $P$ remains unchanged,
whereas the subprincipal changes to
$\subprinc{\tilde{p}} = \subprinc{p} +\{p,\varphi\}$
on $\Sigma$.
Here $\{p,\varphi\}$ denotes the Poisson bracket.
\end{lemma}
Obviously, $P$ is elliptic and bounded from below.
\begin{proof}
The radial derivatives of $p$ and of $\lambda_0=(zv|v)$ are, at $\Sigma$,
equal to $1$ and $(\dot{z}v|v)>0$, respectively.
Therefore, near $\Sigma$, $a_0^2\lambda_0=p-1$
for some $a_0\in\Cinfty$, $a_0>0$.
Set $\tilde{Z}=A_0^*V^*(Z+B)VA_0$, $A_0=\Op_\h(a_0)$.
Choose $\tilde{P}_1\in\Pskm{0}{1}$ (formally) selfadjoint with
leading symbol $p-i\h\trace(\verderiv\horderiv p)/2$.
The selfadjoint operators  $\tilde{Z}$ and $\tilde{P}_1-1$ have the same principal symbol, $p-1$.
Therefore, the imaginary parts of their leading symbols are equal.
It follows that the principal symbol $q_0$ of
$\tilde{Q}_0=\tilde{Z}-(\tilde{P}_1-1)\in\Pskm{0}{0}$
equals, on $\Sigma$, $a_0^2={(\dot{z}v|v)}^{-1}$ times
the real part of the coefficient of $\h$ in \eqref{Z-itwine-mainsymb}.

Define $p_0\in\Cinfty(T^*X\setminus 0)$, homogeneous of degree $0$,
and $r_{-1}\in\Sym{-1}$ such that $q_0=p_0+2(p-1)r_{-1}$ holds in a neighbourhood of $\Sigma$.
Then
\[
(1-\h\Op_\h(r_{-1})^*)\tilde{Z}(1-\h\Op_\h(r_{-1})) = \tilde{P}_1+\h\tilde{P}_0 -1+\h \tilde{Q}_{-1},
\]
where $\tilde{P}_0$ is selfadjoint with principal symbol $p_0$.
Proceeding inductively, we obtain selfadjoint operators $\tilde{P}_j\in\Pskm{0}{j}$
with classical symbols such that, for $N<1$,
\[
(1-\h R_N^*)\tilde{Z}(1-\h R_N) = \h \sum_{N<j\leq 1}\h^{-j} \tilde{P}_j -1 +\h^{-N} \tilde{Q}_N,
\]
where $\tilde{Q}_N\in\Pskm{0}{N}$, $R_N\in\Pskm{0}{-1}$.
Therefore, there is an $\h$-independent operator $P\in\Psmphg{1}$ such that
$(1-\h R^*)\tilde{Z}(1-\h R) = \h P -1$ near $\Sigma$.
Moreover, $P\equiv \tilde{P}_1+\h\tilde{P}_0$ modulo $\Pskm{-2}{-1}$.
The symbol of $P$ equals $p-i\trace(\verderiv\horderiv p)/2 +p_0$ modulo $\Sym{-1}$.
It follows from Corollary~\ref{cor-subprinc-symbol},
or rather its analogue for $\h$-independent operators,
that $p$ is the principal symbol of $P$ and
$\subprinc{p}=p_0$ its subprincipal symbol.
By construction $p_0=q_0$ on $\Sigma$.
Formula \eqref{subprinc-P} follows from the formula for $q_0$ mentioned earlier.

Note $\{p,\varphi\} = \trace\big(\verderiv p.\horderiv \varphi - \horderiv p.\verderiv \varphi\big)$.
The last assertion of the lemma follows from \eqref{subprinc-P},
using $v^*\verderiv z.v= \verderiv \lambda_0$.
\end{proof}

\begin{proof}[Proof of Theorem~\ref{thm-reduct}]
The following assertions hold microlocally in a neighbourhood of $\Sigma$.
It follows from lemmas~\ref{lemma-diagonalize} and \ref{lemma-P}
that, if $A^{-*}$ denotes a parametrix of $A^*$,
$(\Id-R^*)Z(\Id-R)VA =VA^{-*}(\h P -1)$.
Define $J_\h=(\Id-R)VA$ and $\tilde{J}_\h=(\Id-R^*)^{-1} VA^{-*}$.
We have $J_\h,\tilde{J}_\h\in\Pskm{0}{0}$, $\tilde{J}_\h-J_\h\in\Pskm{-1}{-1}$.
Moreover, $J_\h^* J_\h$ is elliptic.
By definition of $Z$, $TB_\h J_\h=\tilde{J}_\h (P-\h^{-1})$,
where $B_\h$ is the Dirichlet parametrix given in Proposition~\ref{exist-dirichlet-param}.
Combining the results in section~\ref{sect-parametrix}
with lemmas~\ref{lemma-diagonalize} and \ref{lemma-P},
the theorem follows.
\end{proof}

\section{Construction of Quasimodes}
\label{sect-modes}

Given $P$ of Theorem~\ref{thm-reduct} we associate to the sequence
of positive eigenvalues of $P$ a sequence of quasimodes of $L_T$.
We follow \cite[sect.~4]{Stefanov00LowerBd}, differing in some details, however.

Let $P$, $B_\h$, and $J_\h$ as in Theorem~\ref{thm-reduct}.
Assume given a sequence of quasimodes, $(\mu_j)$, with almost orthogonal quasimodes states:
\begin{equation}
\label{qm-of-P}
Pf_j-\mu_j f_j=\bigoh_{\Cinfty}(\h_j^{\infty}),
\quad
(f_j|f_k) -\delta_{jk}=\bigoh((\h_j+\h_k)^{\infty}),
\end{equation}
$f_j\in\Cinfty(X;\halfdens)$, $0<\mu_j\leq\mu_{j+1}\to\infty$, $\h_j=\mu_j^{-1}$.

We define quasimode states for the traction-free boundary problem.
By Theorem~\ref{thm-reduct} the traction
$t_j= T B_{\h_j} J_{\h_j} f_j=\bigoh_{\Cinfty}(\h_j^\infty)$.
Choose $u_j'=\bigoh_{\Cinfty}(\h_j^\infty)$ satisfying
$A_0(0)\partial_r u_j'\restrict{X}+t_j=0$ and $u_j'\restrict{X}=0$.
Define $u_j\in\Ccinfty(M;\CT M)$,
\begin{equation}
\label{QM-uj}
u_j=\h_j^{-1/2} \big(B_{\h_j} J_{\h_j} f_j + u_j').
\end{equation}
By Theorem~\ref{thm-reduct},
\begin{equation}
\label{QM-uj-of-LT}
L u_j-\mu_j^{2} \rho u_j =\bigoh_{\Cinfty}(\h_j^{\infty}),
\quad Tu_j =0,
\end{equation}
and $\|u_j\|_{L^2}=\bigoh(1)$.
We can assume that the $u_j$ are supported in a given neighbourhood of $X$.
Using the ellipticity of $L$, we deduce $\|u_j\|_{H^2}=\bigoh(\h_j^{-2})$.

To go from quasimodes to eigenvalues or, in scattering theory, to resonances,
it is desirable to be able to decompose the quasimodes into well-separated clusters.
In addition, the quasimode states of each cluster should be
linearly independent, and remain so after applying small perturbations.

\begin{prop}
\label{prop-qm-indep}
Let the assumptions of Theorem~\ref{thm-reduct} hold.
Assume given quasimodes $\mu_j=\h_j^{-1}>0$ of $P$ as in \eqref{qm-of-P},
and define $u_j$ as in \eqref{QM-uj}.
Then \eqref{QM-uj-of-LT} holds.
Let $m>\dim X$.
There exist $\delta>0$ and a covering of $\{\mu_j\}$
by a sequence of intervals $[a_k,b_k]\subset\R_+$, such that
\[
b_k+2\delta b_k^{-m-\dim X}<a_{k+1}, \quad b_k-a_k<b_k^{-m}.
\]
Let $w_j\in H^2(M;\CT M)$ such that, for some $N\geq 0$,
\[
\|w_j\|_{H^2}=\bigoh(\h_j^{-2-N}), \quad
w_j - u_j=\bigoh_{L^2}(\h_j^{2\dim X+N}).
\]
Then, for large $k$,
$\{w_j\}_{a_k\leq\mu_j\leq b_k}$
is linearly independent.
\end{prop}
\begin{proof}
Property \eqref{QM-uj-of-LT} is clear by the arguments already given.

It is well-known that a quasimode sequence \eqref{qm-of-P}
is asymptotic to a subsequence of the sequence of eigenvalues of $P$.
The latter satisfies the Weyl asymptotics.
Hence we have a Weyl estimate $j\leq C\mu_{j}^{\dim X}$.
It follows that every interval $[a,b]$, $1\leq b$, of length $> L$
has a subinterval of length $\geq L b^{-\dim X}/C$
which contains no quasimode $\mu_j$.
The existence of intervals $[a_k,b_k]$ having the stated
properties follows from this observation.
Compare \cite[Proof of Theorem 2]{Stefanov99Quasimodes}.
Define the set of indices of the $k$-th cluster: $I_k=\{j\setof \mu_j\in[a_k,b_k]\}$.

Choose a left inverse $K_\h\in\Pskm{0}{0}(X;\CT_X M,\halfdens)$
of $J_\h$, $K_\h J_\h=\Id$ at $\Sigma$.
Since $J^*_\h J_\h$ is elliptic at $\Sigma$, $K_\h$ is readily found.

Denote $\gamma:v\mapsto v\restrict{X}$ the trace map.
By \eqref{QM-uj}, $\h_j^{1/2}\gamma u_j= J_{\h_j} f_j+\gamma u_j'$.
From \eqref{qm-of-P} it follows that $\WF_{\h_j} f_j\subset\Sigma$.
Therefore,
\[
\h_j^{1/2} K_{\h_j} \gamma u_j = f_j + \bigoh_{\Cinfty}(\h_j^{\infty}).
\]
By the remark after Lemma~\ref{def-geom-Oph} we can assume that
there exists a constant $C$ such that for all $j,\ell\in I_k$, $k\in\N$,
\[
\|K_{\h_\ell} -K_{\h_j}\|_{L^2\to L^2}\leq C b_k |\h_\ell -\h_j|.
\]
Using $b_k |\h_\ell -\h_j| \leq b_k a_k^{-2}|\mu_\ell-\mu_j| \leq a_k^{-2} b_k^{-m+1}$,
it follows that
\[
\h_j^{1/2} \|(K_{\h_\ell} -K_{\h_j})\gamma u_j\|_{L^2}= \bigoh(b_k^{-m}),
\quad j,\ell\in I_k,
\]
if $k$ is sufficiently large.
The assumptions on $w_j$ imply $\|w_j-u_j\|_{H^1}= \bigoh(\h_j^{1+\dim X})$.
Here we use the estimate $\|v\|_{H^1}^2\leq C\|v\|_{L^2}\|v\|_{H^2}$.
Applying the trace theorem,
$\|\gamma w_j- \gamma u_j\|_{L^2} =\bigoh(\h_j^{1+\dim X})$.
Summarizing the estimates, we have shown that,
for some $\varepsilon>0$,
\[
\|\h_j^{1/2} K_{\h_\ell} \gamma w_j -f_j\|_{L^2} = \bigoh(\h_{\ell}^{\varepsilon+\dim X}),
\quad j,\ell\in I_k.
\]
Because of almost orthogonality of the $f_j$ and the Weyl estimate,
we can apply \cite[Lemma 4]{Stefanov99Quasimodes}.
We obtain, for every $\ell\in I_k$, the linear independence of
$\{K_{\h_\ell} \gamma w_j\}_{j\in I_k}$ when $k$ is sufficiently large.
Since $K_{\h_\ell} \gamma$ is linear, also
$\{w_j\}_{j\in I_k}$ is linearly independent.
\end{proof}

\begin{proof}[Proof of Corollary~\ref{cor-ev-asymp}]
We apply Proposition~\ref{prop-qm-indep} with $\mu_j\uparrow\infty$
the sequence of positive eigenvalues of $P$, counted with multiplicities,
and $\{f_j\}$ a corresponding orthonormal system of eigenvectors.
Fix $m>\dim X$.
Let $[a_k,b_k]$ be the intervals, clustering $\{\mu_j\}$, given in the proposition.
The quasimode states defined in \eqref{QM-uj} belong to the domain
of the selfadjoint operator $L_T$.
Let $\pi_k$ denote the spectral projector for $L_T$ of the interval $[a_k',b_k']$,
where $a_k'=a_k-\delta b_k^{-m-\dim X}$, $b_k'=b_k+\delta b_k^{-m-\dim X}$.
The intervals $[a_k',b_k']$ are pairwise disjoint.
Set $w_j=\pi_k u_j$ if $\mu_j\in[a_k,b_k]$.
A well-known argument, using the spectral theorem, gives
\[
\delta^{2}b_k^{-2m-2\dim X}\|w_j-u_j\|^2_{L^2}\leq \|(L_T-\mu_j^2)u_j\|^2_{L^2}
=\bigoh(b_k^{-\infty})
\]
if $\mu_j\in[a_k,b_k]$.
Since $L_T$ is elliptic, we have $\|w_j\|_{H^2}=\bigoh(\mu_j^2)$.
Now Proposition~\ref{prop-qm-indep}, with $N=0$, implies that, for $k$ sufficiently large,
the rank of $\pi_k$ equals $\sharp\{j\setof\mu_j\in[a_k,b_k]\}$.
Hence an increase by $n$ of $N_P$ over $[a_k,b_k]$ leads to
an increase $\geq n$ of $N_{L_T}$ over $[a_k',b_k']$.
Taking into account the widths of the intervals, the corollary follows.
\end{proof}

\begin{remark}
The foregoing arguments also apply to give lower bounds
for the counting function of resonances.
In this case, $\pi_k$ is the projector onto the space of resonant states
which correspond to resonances in rectangles $[a_k,b_k]+i[0,s_k]$.
To satisfy the assumptions in Proposition~\ref{prop-qm-indep}
for $w_j=\pi_k u_j$, one establishes resolvent estimates.
See \cite{StefanovVodev96Reson}, \cite{TangZworski98Quasimodes}, \cite{Stefanov99Quasimodes},
and \cite{Stefanov00LowerBd}, for ways from quasimodes to resonances.
The clustering method was developed in this context,
\cite{Stefanov99Quasimodes}, to handle multiplicities appropriately.
Resolvent estimates for anisotropic elastic systems are given in \cite{KawaNaka00Poles}.
\end{remark}

\section{The Isotropic Subprincipal Symbol}
\label{sect-isosubpr}

In this section we assume that the elastic medium is isotropic.
We evaluate the subprincipal symbol of $P$, $\subprinc{p}$,
starting from the general formula \eqref{subprinc-P}.

We continue with Example~\ref{exa-iso-1},
referring to the notation introduced there.
The kernel bundle $\Kernel z$ is a line subbundle of $V$,
the subbundle of $\CT_X M$ spanned by $\nu$, $\xih =\xi/|\xi|$.
Abbreviate \eqref{z-iso-block} and \eqref{iq-iso-block} as follows:
\[
(z)_{11} =
\left[\begin{array}{cc}
\zeta_1 & -i\zeta_2\\
i\zeta_2 & \zeta_3
\end{array}\right],
\quad
(iq)_{11} =
\left[\begin{array}{cc}
\kappa_{11} & -i\kappa_{12}\\
i\kappa_{21} & \kappa_{22}
\end{array}\right].
\]
It will be convenient to use the slownesses relative to the Rayleigh wave speed,
$\sigma_s=c_r/c_s$ and $\sigma_p=c_r/c_p$.
Then $t=\sigma_s^2$, $ut=\sigma_p^2$ on $\Sigma=\{c_r|\xi|=1\}$.
Moreover, we set $\tau_s=(1-\sigma_s^2)^{1/2}$, $\tau_p=(1-\sigma_p^2)^{1/2}$,

We first show how to evaluate $(z_-v|v)$,
$v\in\Kernel z$, $z_-$ as in \eqref{sub-z-of-Z}.

\begin{lemma}
\label{lemma-X-iso}
Set $K=(iq)_{11}$.
Define $Y_j$ by \eqref{Y1}, \eqref{Y2}, and \eqref{Y3}.
Let $X=(x_{jk})$ the selfadjoint $2\times 2$ matrix which is the unique solution of
\begin{equation}
\label{sylv-X}
XK+K^* X = -2 Y_1 -Y_2-Y_2^* + Y_3+Y_3^*.
\end{equation}
Let $v=v_1\nu+v_2\xih\in\Kernel z$.
Then
\[
2\RE(z_-v|v)= x_{11} |v_1|^2 + x_{22} |v_2|^2 + 2\RE x_{12}\bar{v_1}v_2.
\]
\end{lemma}
\begin{proof}
Set $x=z_-+z_-^*$. Then $2\RE(z_-v|v)=(xv|v)$.
By \eqref{sub-z-of-Z}, $x$ satisfies the uniquely solvable
Sylvester equation $x(iq)+(iq)^*x=iy+(iy)^*$,
where $y$ equals the right-hand side of \eqref{sub-z-of-Z}.
Since $q$ leaves $V$ and $V^\perp$ invariant,
$X=(x)_{11}=(x_{jk})$ is the unique solution of \eqref{sylv-X}
provided the right-hand side of the equation equals $\big(iy+(iy)^*\big)_{11}$.
The latter holds if
\begin{equation*}
Y_1= (\trace(S) z +\partial_r z)_{11}, \;
Y_2=\big(a_{1-} iq\big)_{11},\;
Y_3= \big(i\trace \verderiv (iq)^*.a\horderiv iq\big)_{11}.
\end{equation*}
Observe that the $a_{2-}$ term of \eqref{sub-z-of-Z} drops out
because of the skewness of $(ia_{2-})_{11}$.
In the following we derive formulas for $Y_j$.

The basis vectors $\nu$ and $\xih$ do not depend on $r$.
Therefore, $(\partial_r z)_{11}= \partial_r (z)_{11}$.
We obtain
\begin{equation}
\label{Y1}
Y_1 = \trace(S)
\left[\begin{array}{cc}
\zeta_1 & -i\zeta_2\\ i\zeta_2 & \zeta_3
\end{array}\right]
+
\left[\begin{array}{cc}
\partial_r\zeta_1 & -i\partial_r\zeta_2\\ i\partial_r\zeta_2 & \partial_r\zeta_3
\end{array}\right].
\end{equation}

Using Lemma~\ref{lemma-a12-minus} and the remark following it we
obtain a formula for $(a_{1-})_{11}$.
Clearly, $(a_{1-}iq)_{11}=(a_{1-})_{11} (iq)_{11}$.
We derive
\begin{equation}
\label{Y2}
Y_2= 
\left[\begin{array}{cc}
\mu\trace S & \langle\xih,\nabla\mu\rangle \\
\langle\xih,\nabla\lambda\rangle & \mu\trace S  +(\lambda+\mu)\langle\xih, S\xih\rangle
\end{array}\right]
\left[\begin{array}{cc}
\kappa_{11} & -i\kappa_{12}\\
i\kappa_{21} & \kappa_{22}
\end{array}\right].
\end{equation}

It remains to determine $Y_3$.
Fix an orthonormal frame $(\eta_j)$ of $T_X^* M$, $\eta_1=\nu$, $\eta_2=\xih$.
To compute the contraction we use the frame $(\eta_j)_{j\geq 2}$ of $T^* X$,
and the dual frame.
We compute derivatives of
\begin{align*}
iq &= |\xi|\sqrt{1-t}(\Id-\nu\otimes\nu-\xih \otimes\xih) \\
   &\phantom{==} +\kappa_{11} \nu\otimes\nu -i\kappa_{12} \nu\otimes\xih
                 +i\kappa_{21} \xih\otimes\nu +\kappa_{22} \xih\otimes\xih.
\end{align*}
Set $s_{jk}=\langle S\eta_j,\eta_k\rangle$.
A calculation using $\horderiv \nu= S$ and $\horderiv\xih =0$ gives
$\big(\Horderiv{_j} iq\big)_{11} = \Horderiv{_j}(iq)_{11} + s_{j2} |\xi|b^{-1} M$,
$j\geq 2$, where
\[
M=
\left[\begin{array}{cc}
0 & (ut-b)\sqrt{1-t} \\ (ut-b)\sqrt{1-t} & i(ut-t)
\end{array}\right].
\]
Regard the coefficients $\kappa_{jk}$ as functions of $c_s,c_p,|\xi|$.
Then
$\Horderiv{_j}(iq)_{11} = \langle\eta_j, \nabla c_s\rangle K_s
+ \langle\eta_j, \nabla c_p\rangle K_p$,
where $K_s$ and $K_p$ denote the partial derivatives of $(iq)_{11}$
with respect to $c_s$ and $c_p$, respectively.
In particular,
\[
\big(\Horderiv{_2} iq\big)_{11} = \langle\xih, \nabla c_s\rangle K_s
    + \langle\xih, \nabla c_p\rangle K_p + s_{22} |\xi|b^{-1} M.
\]

Define $w_1=[(ut-b)\sqrt{1-t}, -i(b-ut)]$.
The row $k>2$ in $\big(\Horderiv{_j} iq\big)_{21}$ equals $s_{jk} b^{-1}|\xi| w_1$.

The vertical derivative of a function $\kappa$ which, when restricted to
a fiber depends only on $|\xi|$, is given by its radial derivative:
\begin{equation}
\label{vert-radial-deriv}
\Verderiv{_\eta}\kappa =|\xi|^{-1} \langle \xih,\eta\rangle \dot{\kappa}.
\end{equation}
A calculation using $\verderiv \nu= 0$ and $\verderiv\xih =|\xi|^{-1}(\Id-\xih \otimes\xih)$ gives
\[
\big(\Verderiv{_j} iq\big)_{11} = \Verderiv{_j}(iq)_{11}
 =|\xi|^{-1} \delta_{2j} \dot{K},
\quad j\geq 2,
\]
where we have set
\[
\dot{K}=
\left[\begin{array}{cc}
\dot{\kappa_{11}} & -i\dot{\kappa_{12}}\\
i\dot{\kappa_{21}} & \dot{\kappa_{22}}
\end{array}\right].
\]
Define $w_2 = [i(b-t), \sqrt{1-ut}-\sqrt{1-t}]$.
The row $k>2$ in $\big(\Verderiv{_j} iq\big)_{21}$ equals $\delta_{jk}  b^{-1} w_2$.

Denote $A=(a)_{11}=\diag[\lambda+2\mu,\mu]$.
Note that $(a)_{22}$ equals $\mu$ times the unit matrix.
Summing over $j\geq 2$ we derive
\begin{equation}
\label{Y3}
\begin{aligned}
Y_3 &= i{\dot{K}}^* A \big(|\xih|^{-1} \langle\xih, \nabla c_s\rangle K_s
    +  |\xih|^{-1} \langle\xih, \nabla c_p\rangle K_p + s_{22} b^{-1} M\big) \\
 &\phantom{==} + i \mu b^{-2}|\xi| (\trace(S)-s_{22}) w_2^*\otimes w_1,
\end{aligned}
\end{equation}
evaluated at $\Sigma$.
\end{proof}

Denote $v$ the unique unit section of $\Kernel(z-\lambda_0)$
satisfying $(\xih|v)>0$,
\[
v=\gamma^{-1}\big(i\zeta_2\nu+(\zeta_1-\lambda_0)\xih\big),
\]
$\gamma>0$ such that $|v|=1$.
We compute the $v$-dependent terms in the right-hand side of \eqref{subprinc-P}.
\begin{lemma}
\label{lemma-IMs}
On $\Sigma$,
$\IM \trace\horderiv p.\verderiv v^*.v=0$, and
\begin{equation}
\label{im-part-in-psub}
\begin{aligned}
16 \gamma^2 & \IM\trace(v^*\verderiv z.\horderiv v)  \\
  &= m^3 \mu^{-1} c_r^2 \sigma_s^6(4-\sigma_s^2)(2-\sigma_s^2)\big(2\tau_s\dot{\zeta_3}
             - (2-\sigma_s^2)\dot{\zeta_2}\big) s_{22} \\
  &\phantom{=} + 2 m^3 c_r \sigma_s^6(2-\sigma_s^2)(5\sigma_s^2 -4-\sigma_s^4)\trace'(S),
\end{aligned}
\end{equation}
where $\trace'(S)=\trace(S)-s_{22}$, $s_{22}=\langle\xih,S\xih\rangle$,
and $m=\mu|\xi|/b$.
\end{lemma}

\begin{proof}
Set $\gamma_1=\zeta_2/\gamma$ and $\gamma_2=(\zeta_1-\lambda_0)/\gamma$.
We continue to use the frame $(\eta_j)$.
For $j\geq 2$ we have
\begin{align*}
\Verderiv{_j}v^* &=-i\Verderiv{_j}\gamma_1.\nu^*+\Verderiv{_j}\gamma_2.\xih^*
         +|\xi|^{-1}(1-\delta_{2j})\gamma_2 \eta_j^*,\\
\Horderiv{_j}v &= i\Horderiv{_j}\gamma_1.\nu + \Horderiv{_j}\gamma_2.\xih  +i\gamma_1 S\eta_j.
\end{align*}
Note that $\Verderiv{_j}v^*.v$ is real.
Hence $\IM \trace\horderiv p.\verderiv v^*.v=0$.
We need the vertical derivative of $z$.
To compute it we proceed in the same way
as we did when computing the derivatives of $iq$.
Recall that $z$ equals $\zeta^\perp\Id$ on $V^\perp$,
where $\zeta^\perp=\mu|\xi|\sqrt{1-t}$.
We obtain
$\big(\Verderiv{_j} z\big)_{11} = \Verderiv{_j} (z)_{11}$.
Moreover, the column $k>2$ in $(\Verderiv{_j} z)_{12}$
equals $\delta_{jk}|\xi|^{-1}$
times the transpose of the row vector $[-i\zeta_2, \zeta_3-\zeta^\perp]$.
We get
\begin{align*}
\IM v^* \Verderiv{_j} z.\Horderiv{_j} v 
  &= \gamma_1 \RE v^* \Verderiv{_j} z. S\eta_j \\
  &= \gamma_1 (\gamma_2\Verderiv{_j} \zeta_3 -\gamma_1\Verderiv{_j}\zeta_2) s_{2j} \\
  &\phantom{=} + \gamma_1 |\xi|^{-1} \big(\gamma_2(\zeta_3-\zeta^\perp)-\gamma_1\zeta_2\big) s_{jj} (1-\delta_{2j}).
\end{align*}
Summing over $j\geq 2$ we obtain
\begin{align*}
\gamma^2 \IM\trace(v^*\verderiv z.\horderiv v) 
  &= \zeta_1\zeta_2\Verderiv{_{S\xih}} \zeta_3 -\zeta_2^2\Verderiv{_{S\xih}}\zeta_2 \\
  &\phantom{==} + c_r \zeta_2 \big(\zeta_1(\zeta_3-\zeta^\perp)-\zeta_2^2\big) \trace'(S).
\end{align*}
The first term on the right equals
\[
m^2 |\xi|^{-1} s_{22} (2b-t)\big(t\sqrt{1-t}\dot{\zeta_3}-(2b-t)\dot{\zeta_2}\big).
\]
Moreover, using the definition of $b$, we calculate
\[
\zeta_3-\zeta^\perp = m(\sqrt{1-ut}-\sqrt{1-t}).
\]
Using \eqref{b-on-Sigma}, $4b=t(4-t)$, we derive \eqref{im-part-in-psub}.
\end{proof}

The restriction to $\Sigma$ of the radial derivative of the eigenvalue
$\lambda_0=(zv|v)=a_0^{-2}(p-1)$ equals $\dot{\lambda_0}=(\dot{z}v|v)=a_0^{-2}$
because $\dot{p}=1$ on $\Sigma$.

\begin{lemma}
\label{lemma-lambda0-dot}
On $\Sigma$,
\[
\gamma^2 \dot{\lambda_0}
  = m^3 \sigma_s^6(4-\sigma_s^2) \tau_s \big(\tau_p/\tau_s+c_s\tau_s/c_p\tau_p +\sigma_s^2-2\big),
\]
where $m=\mu|\xi|/b$.
\end{lemma}

\begin{proof}
The section $w=i\zeta_2\nu+\zeta_1\xih$ equals $\gamma v$ on $\Sigma$.
Therefore, to second order on $\Sigma$,
$\gamma^2 \lambda_0\equiv (zw|w) = \zeta_1 \det (z)_{11}$.
Inserting \eqref{det-z11},
\[
\gamma^2 \lambda_0\equiv
mb t\sqrt{1-t} \big(4 \sqrt{(1-t)(1-ut)} -(2-t)^2\big).
\]
Recall $p=c_r|\xi| = \sigma_s t^{-1/2}$, $\Sigma=\{t=\sigma_s^2\}$.
The rule of de l'Hospital gives
\begin{equation*}
\lim_{t\to\sigma_s^2}
\frac{4 \sqrt{(1-t)(1-ut)} -(2-t)^2}{\sigma_s t^{-1/2} -1}
   = 4\sigma_s^2\big(\tau_p/\tau_s+c_s\tau_s/c_p\tau_p +\sigma_s^2-2\big).
\end{equation*}
Summarizing, the formula for
$\gamma^2 \dot{\lambda_0} =\gamma^2 \lambda_0/(p-1)$ follows.
\end{proof}

Inserting the formulas of the lemmas of this section
into the general formula~\eqref{subprinc-P} for the subprincipal symbol of $P$
we obtain a formula for the subprincipal symbol in the isotropic case.

\begin{prop}
\label{prop-subprinc-isoelast}
Denote $X=(x_{jk})$ the $2\times 2$ matrix solving \eqref{sylv-X}.
Set
$N= \tau_s \big(\tau_p/\tau_s+c_s\tau_s/c_p\tau_p +\sigma_s^2-2\big)$.
Let $P$ be the operator of Lemma~\ref{lemma-P} determined by the unit section
$v$ of $\Kernel z$ having positive $\xih$ component.
The subprincipal symbol of $P$ is given as follows.
\begin{align*}
16 N \subprinc{p} &= (c_r/2\mu) \big(x_{11} (2-\sigma_s^2)^2 + 4 x_{22} (1-\sigma_s^2)^2
        + 4\IM x_{12}(2-\sigma_s^2)\tau_s\big) \\
  &\phantom{=} + \mu^{-1} c_r^2 (2-\sigma_s^2)\big(2\tau_s\dot{\zeta_3}
             - (2-\sigma_s^2)\dot{\zeta_2}\big) \langle S\xih,\xih\rangle \\
  &\phantom{=} + 2 c_r (4-\sigma_s^2)^{-1} (2-\sigma_s^2)(5\sigma_s^2 -4-\sigma_s^4)
                     (\trace(S)-\langle S\xih,\xih\rangle)
\end{align*}
\end{prop}
\begin{proof}
On $\Sigma$,  $w=\gamma v=(mt/2)\big(i(2-t)\nu+2\sqrt{1-t}\xih\big)$.
Using lemmas~\ref{lemma-X-iso} and \ref{lemma-lambda0-dot} we calculate
$16 N \RE(z_-w|w)/\gamma^2\dot{\lambda_0}$.
The result is the first term on the right-hand side of the claimed formula.
Similarly, we obtain the other terms combining the
lemmas~\ref{lemma-IMs} and \ref{lemma-lambda0-dot}.
\end{proof}

The constituents of the above formula for $\subprinc{p}$ are curvature
and velocities (Lam\'e parameters), assumed known.
It seems difficult to analyze the formula further unless it is specialized to particular cases.
However, it should be noted that the formula allows explicit
numerical evaluation of $\subprinc{p}$.
Therefore it can be used when solving transport equations for Rayleigh
wave amplitudes numerically with a (seismic) ray tracing program, say.
Formulas for the amplitudes of Rayleigh waves were given by 
Babich and Kirpichnikova in \cite{babich/kirpichnikova:04:rayleigh-waveprop}.

\appendix
\section{Geometric \Psdiff~Calculus}
\label{app-geopsdo-calc}

\Psdiff\ operators on manifolds are usually introduced by reducing to
the euclidean case via partitions of unity,
\cite[18.1]{hormander:85:the-analysis-3},
\cite{EvansZworskiXXSemiclass}.
The principal symbol of a \psdiff\ operator is invariantly defined.
If the operator acts on sections of the line bundle of half-densities
then there also is an invariantly defined subprincipal symbol,
\cite[Theorem~18.1.33]{hormander:85:the-analysis-3},
\cite[Appendix]{SjoesZworski02Monodromy}.

In the body of the paper we explicitly track, down to the subprincipal level,
symbols of operators acting between vector bundles.
To achieve this we use Sharafutdinov's geometric \psdiff\ calculus
\cite{Sharaf04geo1,Sharaf05geo2}.
The purpose of this appendix is to recall this calculus,
presenting a semiclassical variant.
Since we have to refer, in the main part of the present paper,
to proofs of the calculus, we give a rather detailed presentation.
The calculus depends on a symmetric connection of the manifold
and on metric connections of the (Hermitian) bundles.
We make the stronger assumption that the manifold is Riemannian
and that the symmetric connection is the Levi-Civita connection.
The important features of the calculus are a symbol isomorphism
modulo order minus infinity, and complete symbol expansions for products and
adjoints given solely in terms of geometric data.
Using connections to develop a \psdiff\ calculus and to prove the existence
of a complete symbol isomorphism was done earlier by Widom, \cite{Widom80completesymb}.
This was further developed by Pflaum who gave a convenient quantization
map from symbols to operators, \cite{Pflaum98normalsymb}.
Sharafutdinov gave symbol expansions in terms of geometric data.

Let $X$ a compact Riemannian manifold without boundary, $\dim X=n$.
The exponential map, $\exp$, of the Levi-Civita connection
defines a diffeomorphism, $(x,v)\mapsto (x,y)=(x,\exp_x v)$,
between a neighbourhood of the zero-section of the tangent bundle $T=TX$
and a neighbourhood of the diagonal in $X^2$.
In the proofs of the propositions below we need the following properties of $\exp$.
In local coordinates the exponential map satisfies
\begin{equation}
\label{exp-coords}
(\exp_x v)^i = x^i + v^i -\Gamma^i_{jk}(x) v^jv^k/2 + \bigoh(|v|^3),
\end{equation}
where $\Gamma^i_{jk}$ denote the Christoffel symbols.
Normal coordinates centered at $x$ satisfy $(\exp_x v)^i =  v^i$.
There exist $0<r<R<\inj(x)$, the injectivity radius of $X$,
such that the equation
\begin{equation}
\label{def-diffeom-f}
\exp_{\exp_x v} z =\exp_x w
\end{equation}
defines, for every $v\in T_x=T_x X$, $|v|<R$, a diffeomorphism $w\mapsto z=z(x,v,w)$
from an open neighbourhood of the origin, contained in $\{|w|<R\}\subset T_x$,
onto the ball $\{|z|<r\}\subset T_y$ , $y=\exp_x v$.
This map is used below to change variables of integration.
Obviously, $z(x,0,w)=w$.
A computation in normal coordinates centered at $x$ shows that
\begin{equation}
\label{deriv-of-z-of-w}
(z_w')^{-1} z = w-v + \bigoh((|v|+|w|)^3) \quad\text{as $v,w\to 0$.}
\end{equation}
Recall, from section~\ref{sect-conn-geosymb}, the notation for
segments and for parallel transport maps.
In local coordinates,
\begin{equation}
\label{transport-coords}
\big(\transport^{TX}_{\tofrom{\exp_x v}{x}} w\big)^i
= w^i -\Gamma^i_{jk}(x) w^jv^k + \bigoh(|v|^2).
\end{equation}

Let $E\ueber X$ and $F\ueber X$ be Hermitian vector bundles with metric connections.
Recall from section~\ref{sect-conn-geosymb} the definition \eqref{def-horderiv}
of horizontal derivatives and the definition of vertical derivatives.
A $\Cinfty$ section $a$ of the bundle $\pi^*\Hom(E,F)\ueber T^* X$
is called a $\Hom(E,F)$-valued symbol of order $m\in\R$,
$a\in\Sym{m}=\Sym{m}(T^* X;\pi^*\Hom(E,F))$,
iff for all nonnegative integers $j$ and $\ell$,
\[
\sup_{x,\xi} (1+|\xi|)^{j-m} |(\verderiv)^j(\horderiv)^\ell a(x,\xi)\big| <\infty.
\]
These are the usual type $1,0$ symbol estimates.
The symbol space $\Sym{m}$ is a Fr\'echet space.
The space $\Sykm{k}{m}=\Sykm{k}{m}(T^* X;\pi^*\Hom(E,F))$
of $\h$-dependent symbols of order $m$ and degree $k$ is the Fr\'echet space of
families $a_\h\in \Sym{m}$ such that $\{\h^k a_\h\setof 0<\h\leq 1\}$ is bounded in $\Sym{m}$.
We call $a\in\Sykm{k}{m}$ classical if there exists an asymptotic expansion 
$a\sim\sum_j h^{j-k}a_j$ with $\h$-independent symbols $a_j\in\Sym{m-j}$.

In the following lemma we define, in a semiclassical setting,
the quantization of symbols according to Sharafutdinov's geometric \psdiff\ calculus.
We relate this definition of $\h$-\psdiff\ operators
to the definition in the euclidean situation.
For semiclassical analysis, in particular,
for the class $\Pskm{k}{m}=\Op_\h \Sykm{k}{m}$ of $\h$-\psdiff\ operators,
including mapping properties, and for frequency sets ($\h$-wavefront sets),
refer to \cite{Gerard88asymptPoles}, \cite{Ivrii98MAPSA},
\cite{dimassi/sjostrand:99:spectral-asympt},
\cite{SjoesZworski02Monodromy}, \cite{EvansZworskiXXSemiclass}.
The class of negligible operators, $\Pskm{-\infty}{-\infty}$,
consists of $\h$-dependent operators whose
Schwartz kernels are $\Cinfty$ with $\bigoh(\h^{\infty})$ seminorms.

Fix $\chi_0\in\Cinfty(TX)$, real-valued, $|v|<r$ on the support of $\chi_0(x,v)$,
such that $\chi_0=1$ in a neighbourhood of the zero-section in $TX$.
\begin{lemma}
\label{def-geom-Oph}
Let $a_\h\in\Sykm{k}{m}$ be a $\Hom(E,F)$-valued symbol.
Then
\begin{equation}
\label{def-Oph-a}
\begin{aligned}
A_\h u_\h(x)=
   (2\pi\h)^{-n} \int \limits _{T^*_x} \int \limits _{T_x}
    & e^{-i\langle\eta,v\rangle /\h} \chi_0(x,v) \\
    & \cdot a_\h(x,\eta) 
    \transport^{E}_{\tofrom{x}{\exp_x v}} u_\h(\exp_x v)\intd v \intd\eta,
\end{aligned}
\end{equation}
defines an $\h$-\psdiff\ operator $A_\h\in\Pskm{k}{m}(X;E,F)$.
Given a point $x$ there exists a geodesic ball $U$ centered at $x$, and
a symbol $a^U_\h\in\Sykm{k}{m}$ such that, for $u_\h$ compactly supported in $U$,
\[
A_\h u_\h(y) = (2\pi\h)^{-n} \int \limits _{T^*_x} \int \limits _{T_x}
             e^{i\langle\theta,v-w\rangle /\h} a^U_\h(y,\theta) 
             \transport^{E}_{\tofrom{y}{y'}} u_\h(y')\intd w \intd\theta,
\]
where $y=\exp_x v$ and $y'=\exp_x w$.
Moreover, at $x$, $a^U_\h\equiv a_\h$ modulo $\Sykm{k-2}{m-2}$.
Every $\h$-\psdiff\ operator is, modulo negligible operators,
of the form \eqref{def-Oph-a}.
\end{lemma}
The measures in \eqref{def-Oph-a} are the normalized
Lebesgue measures of the euclidean spaces $T_x$ and $T^*_x$.
\begin{proof}
We shall drop the subscript $\h$ from the notation.
Fix $x\in X$. Let $U$ denote a geodesic ball with center $x$ and radius $\leq R$.
In the following we assume that the support of $u$ is a compact subset of $U$.
In \eqref{def-Oph-a} we replace the variables $x,v,\eta$ by $y,z,\zeta$.
Next we change variables in the integral $A u(y)$ such that the domain of
integration does not depend on $y$.
Set $y=\exp_x v$.
Define $z=z(x,v,w)$ by \eqref{def-diffeom-f}.
Using the symplectic map 
$(w,\vartheta)\mapsto (z,\zeta)$, $\zeta = {}^t(z_w')^{-1}\vartheta$,
we get
\[
A u(y) = \int \limits _{T_x} K(v,w)
             \transport^{E}_{\tofrom{y}{\exp_x w}} u(\exp_x w)\intd w,
\]
where the kernel $K$ is given by
\[
K(v,w)= (2\pi\h)^{-n} \int \limits _{T^*_x} 
             e^{-i\varphi /\h} \chi_0(y,z)  a(y,\zeta) \intd\vartheta,
\]
$\varphi=\langle\zeta,z\rangle=\langle \vartheta, (z_w')^{-1}z\rangle$.
Since $z=0$ if and only if $v=w$, we have
$\varphi(v,w,\vartheta)= \langle \psi(v,w)\vartheta,w-v\rangle$.
Here $\psi=\Id+\bigoh(|v|^2+|w|^2)$ by \eqref{deriv-of-z-of-w}.
Decreasing the radius of $U$ and making the linear
change of variables $\theta=\psi(v,w)\vartheta$, we get
\[
K(v,w)= (2\pi\h)^{-n} \int \limits _{T^*_x} 
             e^{i\langle\theta,v-w\rangle/\h} \chi_0(y,z)  a(y,\zeta) J_1(v,w) \intd\theta,
\]
$J_1(v,w)=1+\bigoh(|v|^2+|w|^2)$.
It follows that $A$ restricted to $U$ is a $\h$-\psdiff\ operator of class $\Pskm{k}{m}$.
As it stands the symbol depends on $v,\theta,w$.
Using the standard symbol reduction procedure we obtain $a^U(\exp_x v,\theta)$.
Moreover, the asymptotic expansion implies that, at $v=0$, $a^U-a\in\Sykm{k-2}{m-2}$.

Note that $A u(y)=0$ if the distance between $y$ and $\supp u$ is $>r$.
Using a partition of unity, we infer that the class of operators
given by \eqref{def-Oph-a} equals the class of $\h$-\psdiff\ operators
with Schwartz kernels supported in small neighbourhoods of the diagonal.
\end{proof}

Standard arguments show that up to a negligible operator
$A_\h=\Op_h(a_\h)$ does not depend on the choice of the cutoff $\chi_0$.
The space
\[
\Pskm{k}{m}(X;E,F) = \Op_\h \Sykm{k}{m} + \Pskm{-\infty}{-\infty}.
\]
is the space $\h$-\psdiff\ order $m$ and degree $k$.
We denote the geometric symbol $\sigma_\h(A_\h)=a_\h$.

\begin{remark}
Let $A_\h=\Op_\h(a_\h)\in \Pskm{0}{0}$.
Then $A_\h$ is $L^2$ bounded, uniformly in $\h$.
Assume, in addition, that $a_\h$ depends differentiably on $\h$ with $\partial_\h a_h\in \Sykm{0}{0}$.
Changing variables in \eqref{def-Oph-a} from $\eta$ to $\xi=\eta/\h$, we obtain
$A_{\h_1}-A_{\h_0}=\int_{\h_0}^{\h_1} \h^{-1} \Op_\h(b_\h)\intd \h$,
where $b_\h\in\Sykm{0}{0}$,
$b_\h(x,\eta)=\h\partial_{\h}a_\h(x,\eta)+\Verderiv{_\eta} a_\h(x,\eta)$.
This implies the following useful Lipschitz estimate:
\[
\|A_{\h_1} -A_{\h_0}\|_{L^2\to L^2}\leq C h_0^{-1} |\h_1 -\h_0|
\quad\text{if $\h_0<\h_1$,}
\]
where $\|\Op_\h(b_\h)\|_{L^2\to L^2}\leq C<\infty$.
The assumption holds if $a_\h$ is classical and given as a Borel sum.
\end{remark}

In the following, we often suppress from writing the $\h$-dependence
of symbols, operators and distributions.
Moreover, when dealing with integrals like \eqref{def-Oph-a},
we move, without explicitly writing this, the $x$-dependency from the domain of integration
into the integrand using arguments as in the proof of the lemma.

Lebesgue measure $\intd v$ on $T_x X$ and Riemannian volume
are related by $\int f(y)\dvol_X(y)=\int f(\exp_x v)J_0(x,v)\intd v$, $y=\exp_x v$,
with Jacobian $J_0=1+\bigoh(|v|^2)$ at $v=0$.
Let $A=A_\h$ as in \eqref{def-Oph-a}.
The Schwartz kernel $K_A$ of $A$,
\[
Au(x)=\int_X K_A(x,y)u(y)\dvol_X(y), \quad K_A(x,y)\in\Hom(E_y,F_x).
\]
equals in a neighbourhood of the diagonal
a partial Fourier transform of the symbol,
\begin{equation}
\label{kernel-of-A}
K_A(x,y)= (2\pi\h)^{-n} \int \limits _{T^*_x}
     e^{-i\langle\eta,\exp_x^{-1} y\rangle /\h}
     a(x,\eta) \intd\eta \; \psi(x,y) \transport^{E}_{\tofrom{x}{y}}.
\end{equation}
Here $\psi(x,y)=\chi_0(x,v)/J_0(x,v)$, $y=\exp_x v$.
The symbol $a$ is recovered via the inverse Fourier transform:
\begin{equation}
\label{symbol-from-kernel}
a(x,\xi)\equiv \int \limits _{T_x} e^{i\langle\xi,v\rangle /\h}
      (\chi_0 J_0)(x,v) K_A(x,\exp_x v) \transport^{E}_{\tofrom{\exp_x v}{x}} \intd v
\end{equation}
modulo $\Sykm{-\infty}{-\infty}$.
The correspondence between an operator $A=\Op_\h(a)$ and its full symbol $a$,
named the geometric symbol of $A$, defines the complete symbol isomorphism
\[
\Pskm{k}{m}(X;E,F)/\Pskm{-\infty}{-\infty}
  \cong \Sykm{k}{m}(T^* X;\Hom(\pi^* E,\pi^* F))/\Sykm{-\infty}{-\infty}.
\]
The geometric symbol can also be computed by applying the operator to suitable testing functions as follows.
\begin{equation}
\label{testing-geom-symb}
a(x,\xi)s \equiv A_y\big(e^{i\langle\xi,\exp_x^{-1} y\rangle/\h}
               \chi_0(x,\exp_x^{-1} y) \transport^E_{\tofrom{y}{x}} s\big)\restrict{y=x}.
\end{equation}
Here $A_y$ means that $A$ acts on functions of the variable $y$.
In particular, in case $E=\C$, the geometric symbol is obtained at the center of
normal coordinates $x^j$ when $A$ is applied to $e^{i \xi_j x^j/\h}$ and evaluated at $x^j=0$.

We derive symbol properties and expansions 
using the method of stationary phase:
\begin{align*}
\big(\det (H/2\pi i\h)\big)^{1/2}
 & \int e^{i (\varphi(x))/\h} a(x)\intd x \\
 & = \exp\big(2^{-1}i\h\langle H^{-1} \partial,\partial\rangle\big)
           \big(e^{i\rho(x)/\h}a(x)\big)\Restrict{x=0} \\
 & = \sum_{j<3N} \frac{(i\h)^j}{j! 2^j} \langle H^{-1} \partial,\partial\rangle^j
           \big(e^{i\rho(x)/\h}a(x)\big)\Restrict{x=0} +\bigoh(\h^N),
\end{align*}
if $\varphi\in\Cinfty$, real-valued, $\varphi'(x)=0$ iff $x=0$,
$H=\varphi''(0)$ non-singular, $\varphi(0)=0$.
The remainder $\rho(x)=\varphi(x)-\langle Hx,x\rangle/2$ vanishes to third order at $x=0$.
The expansion has the advantage, when compared to that
obtained using the Morse lemma, of giving an efficient algorithm for computing the asymptotic series.
See \cite[Theorem~7.7.5]{hormander:83:the-analysis-1} where the
expansion is arranged in powers of $\omega^{-1}=\h$.

We are mainly interested in the leading symbols of operators.
We call the residue of $a$ in $\Sykm{k}{m}/\Sykm{k-2}{m-2}$
the leading symbol of an operator $\Op_\h(a)\in\Pskm{k}{m}$.
The principal symbol is, of course, the residue in $\Sykm{k}{m}/\Sykm{k-1}{m-1}$.

\begin{prop}
\label{geocalc-adjoint}
Let $A=\Op_\h(a)$ as in \eqref{def-Oph-a} with geometric symbol $a\in\Sykm{k}{m}$.
The formal adjoint $A^*\in\Pskm{k}{m}(X;F,E)$ has the geometric symbol
\begin{equation}
\label{geosymb-adjoint}
b\equiv a^* - i\h \trace \verderiv\horderiv a^* \mod \Sykm{k-2}{m-2}.
\end{equation}
If $a$ is classical then so is $b$.
\end{prop}
Notice that $\verderiv\horderiv a^*$ is a section of
$\pi^*(\Hom(F,E)\otimes T\otimes T^*)$.
The trace is taken of the $T\otimes T^*$ part.

\begin{proof}
The formal adjoint of $A$ is defined by
\[
\int_X \big(u_1(x)|Au_2(x)\big)_F\dvol_X(x) = \int_X \big(A^*u_1(y)|u_2(y)\big)_E\dvol_X(y).
\]
The Schwartz kernel satisfies $K_{A^*}(x,y) = {K_A(y,x)}^*$.
Recall that parallel transport preserves inner products.
It follows from \eqref{kernel-of-A} that
\[
K_{A^*}(x,y)= (2\pi\h)^{-n} \int \limits _{T^*_y}
     e^{i\langle\eta,\exp_y^{-1} x\rangle /\h}
      \transport^{E}_{\tofrom{x}{y}} a(y,\eta)^* \intd\eta \; \psi(y,x),
\]
and $K_{A^*}(x,y)=0$ if the distance between $x$ and $y$ is $>r$.
Set $y=\exp_x v$. Define $z\in T_y$ by $\exp_y z=x$.
After a linear change variables from $\eta\in T^*_y$ to $\zeta={}^t(\exp_x'(v))\eta\in T^*_x$
we have
\[
K_{A^*}(x,y)= (2\pi\h)^{-n} \int \limits _{T^*_x}
     e^{i\langle\eta, z \rangle /\h}
      \transport^{E}_{\tofrom{x}{y}} a(y,\eta)^* \intd\zeta \; \psi(y,x)/J_1(x,v),
\]
with Jacobian $J_1(x,v)=1+\bigoh(|v|^2)$.
Define
\[
b(x,\xi) = \int \limits _{T_x} e^{i\langle\xi,v\rangle /\h}
         (\chi_0J_0)(x,v) K_{A^*}(x,y) \transport^{F}_{\tofrom{y}{x}} \intd v.
\]
Inserting $K_{A^*}$ we have
\begin{equation}
\label{symb-b-of-Astar}
b(x,\xi) = (2\pi\h)^{-n} \int \limits _{T_x} \int \limits _{T^*_x}
	e^{i\varphi/\h} \tilde{a} J \intd \zeta \intd v,
\end{equation}
where
\begin{align*}
\varphi &= \langle\xi,v\rangle + \langle\eta, z \rangle
                = -\langle\zeta-\xi,v\rangle + \langle\zeta, \Phi \rangle, \\
\tilde{a} &= \transport^E_{\tofrom{x}{y}}
       a(y,\eta)^* \transport^F_{\tofrom{y}{x}}
       = \transport^{\Hom(F,E)}_{\tofrom{x}{y}} a(y,\eta)^*, \\
J &=\chi_0(x,v)J_0(x,v) \psi(y,x)/J_1(x,v)=1+\bigoh(|v|^2),
\end{align*}
and $\Phi=\Phi(x,v)={\exp_x'(v)}^{-1}z +v$.
A computation in normal coordinates centered at $x$ shows that 
$\Phi=\bigoh(|v|^3)$ as $v\to 0$.
If $\varphi_{\zeta}'=0$ then $z=0$, hence $v=0$.
It follows that the critical points of $\varphi$ are defined by $v=0$, $\zeta=\xi$.

Apply the method of stationary phase to \eqref{symb-b-of-Astar} and deduce $b\in\Sykm{k}{m}$.
Moreover, the following asymptotic expansion holds:
\begin{equation}
\label{expansion-adjoint}
b\sim \sum_j \frac{(i\h)^j}{j!} \langle -\partial_\zeta,\partial_v\rangle^j
        \big(e^{i\langle \zeta, \Phi\rangle /\h} \tilde{a}\big)\Restrict{v=0,\zeta=\xi}.
\end{equation}
Differentiation of the exponential factor brings out a non-zero factor
only if it comsumes at least three derivatives with respect to $v$
and at most one derivative with respect to $\zeta$.
It follows that the sum is asymptotic.
Moreover, $b$ is determined
modulo $\Sykm{k-2}{m-2}$ by the terms in the asymptotic sum with $j<2$,
$b\equiv a^* -i\h \langle \partial_\zeta,\partial_v\rangle \tilde{a}$.
Observe that
\[
\transport^{T}_{\tofrom{x}{\exp_x v}}\circ \exp_x'(v)=\Id_{T_x} +\bigoh(|v|^2)
\quad\text{as $v\to 0$.}
\]
It follows that $\partial_v\tilde{a}\Restrict{v=0}=\horderiv a^*(x,\zeta)$.
Hence $b\equiv a^* - i\h \trace \verderiv\horderiv a^*$.
The Schwartz kernels of $\Op_\h(b)$ and $A^*$ are equal in a neighbourhood of the diagonal.
Therefore $A^*-B\in\Pskm{-\infty}{-\infty}$. 
\end{proof}

\begin{prop}
\label{geocalc-compos}
Let $A\in\Pskm{k_A}{m_A}(X;F,G)$ and $B\in\Pskm{k_B}{m_B}(X;E,F)$
with geometric symbols $a$ and $b$, respectively.
Set $k=k_A+k_B$, $m=m_A+m_B$.
Then $AB\in\Pskm{k}{m}(X;E,G)$ with geometric symbol
\begin{equation}
\label{geo-symb-compos}
c\equiv ab - i\h \trace\big(\verderiv a.\horderiv b\big)
\end{equation}
modulo $\Sykm{k-2}{m-2}$.
If $a$ and $b$ are classical then so is $c$.
\end{prop}
Again the trace is taken of the $T\otimes T^*$ part,
and the dot terminates differentiated expression.
\begin{proof}
Setting $y=\exp_x v$, $C=AB$ is given by
\begin{align*}
Cu(x)= (2\pi\h)^{-2n} & \iiiint \limits _{T_x\times T^*_x\times T_y\times T^*_y}
     e^{-i(\langle\eta,v\rangle+\langle\zeta,z\rangle) /\h} a(x,\eta) \\
   &\phantom{==} \cdot \transport^F_{\tofrom{x}{y}}\big(b(y,\zeta)
           \transport^E_{\tofrom{y}{\exp_y z}} u(\exp_y z)\big)
           \intd z\intd\zeta\intd v\intd\eta.
\end{align*}
Here and in the following we do not write the cutoff factors.
Let $z=z(x,v,w)$ be the solution of $\exp_y z=\exp_x w$.
The symplectic change of variables
$(w,\vartheta)\mapsto (z,\zeta)$, $\zeta = {}^t(z_w')^{-1}\vartheta$,
preserves the volume form.
We get
$Cu(x)=\int_{T_x} K_C(x,\exp_x w)u(\exp_x w) J_0(x,w) \intd w$,
with Schwartz kernel
\begin{align*}
K_C (x, & \exp_x w) J_0(x,w) \\
  &= (2\pi\h)^{-2n} \int \limits _{T^*_x\times T_x\times T^*_x}
       e^{-i(\langle\eta,v\rangle+\langle \zeta,z\rangle) /\h}
         c_0 \intd(\vartheta,v,\eta) \; \transport^E_{\tofrom{x}{\exp_x w}},
\end{align*}
$c_0 = a(x,\eta) \transport^{\Hom(E,F)}_{\tofrom{x}{y}} b(y,\zeta) M(x,w,v)$.
Here $M(x,w,v)\in GL(E_x)$ denotes the parallel transport in $E$ along
the geodesic triangle $x\to \exp_x w \to \exp_x v \to x$.
It follows that the symbol of $C$ equals
\begin{equation}
\label{c-statphase}
c(x,\xi)
  = (2\pi\h)^{-2n} \int \limits _{T_x\times T^*_x\times T_x\times T^*_x}
    e^{i\varphi/\h} c_0 \intd(v,\eta,w,\vartheta),
\end{equation}
$\varphi =\langle\xi,w\rangle-\langle\eta,v\rangle-\langle \zeta,z\rangle$.
We introduce $w-v$ as a new variable, $w$.
Then \eqref{c-statphase} holds with
\begin{align*}
\varphi &= -\langle\eta-\xi,v\rangle-\langle\vartheta-\xi,w\rangle
                +\langle \vartheta, \Phi\rangle, \\
c_0 &= a(x,\eta) \transport^{\Hom(E,F)}_{\tofrom{x}{y}} b(y,\zeta) M(x,w+v,v),
\end{align*}
Here $\Phi=w- (z_w'(x,v,w+v))^{-1}z(x,v,w+v)\in T_x^*$.
By \eqref{deriv-of-z-of-w}, $\Phi$ vanishes to third order at $v=w=0$.
Clearly, $v=0=z$ at a critical point of $\varphi$.
It follows that $v=w=0$ and $\eta=\vartheta=\xi$ define the critical points.

Now apply the method of stationary phase to \eqref{c-statphase} and deduce
that $c\in\Sykm{k}{m}$ is a symbol which, moreover, has an asymptotic expansion
\begin{equation}
\label{expansion-compos}
c\sim \sum_j \frac{(-i\h)^j}{j!} \big(\langle \partial_\vartheta,\partial_{w}\rangle 
    +\langle \partial_\eta,\partial_v\rangle\big)^j
        \big(e^{i\langle \vartheta, \Phi\rangle /\h} c_0\big)\Restrict{v=w=0,\eta=\vartheta=\xi}.
\end{equation}
Using that $\Phi$ does not depend on $\eta$ and $\vartheta$, and vanishes to third order
at $v=w=0$, we infer that the summands with $j>1$ belong to $\Sykm{k-2}{m-2}$.
It follows that
\[
ab - i\h \langle \partial_\eta a, \partial_v \tilde{b} M\rangle
 -i\h a \langle \partial_\vartheta,\partial_{w}\rangle \tilde{b} M,
\]
evaluated at the critical point, is the leading symbol of $C$.
Here $\tilde{b}=\transport^{\Hom(E,F)}_{\tofrom{x}{y}} b(y,\zeta)$.
We have $\partial_{w}\tilde{b}=0$ at $v=w=0$.
This follows from $\zeta_w'=0$ which is a corollary of $z=w$ at $v=0$.
The derivatives of $M$ with respect to $v$ and $w$ vanish at $v=w=0$.
Using
$\transport^{T}_{\tofrom{x}{\exp_x v}}\circ z_w'=\Id_{T_x} +\bigoh(|v|^2)$
at $w=0$, we derive
\[
\partial_{v}\tilde{b}
  = \partial_{v}
       \transport^{\Hom(E,F)}_{\tofrom{x}{\exp_x v}} b(\exp_x v,{}^t(z_w')^{-1}\vartheta)
  = \horderiv b(x,\vartheta),
\]
at $v=w=0$.
Summarizing the computations, \eqref{geo-symb-compos} follows.
\end{proof}

\begin{remark}
The proofs of propositions~\ref{geocalc-adjoint} and \ref{geocalc-compos}
follow those in \cite{Sharaf04geo1,Sharaf05geo2} closely with only minor modifications.
Our derivation of the asymptotic expansions of the symbols of adjoints
and products may be somewhat shorter, however.
We differ in defining the adjoint with respect to the volume element
rather than by using half-densities.
Notice that the symbol expansions 
\eqref{expansion-adjoint} and \eqref{expansion-compos}
depend only on the given symbols and on the geometry.
In the formulas \eqref{geosymb-adjoint} and \eqref{geo-symb-compos},
we extracted the leading symbols.

For the purposes of the present paper it suffices to assume $X$ compact.
A symbol calculus on general (complete) Riemannian manifolds 
needs to take the injectivity radius into account
and handle mapping properties more explicitly.
\end{remark}

It is well-known that a \psdiff\ operator acting on half-densities
has an invariantly defined subprincipal symbol;
see \cite[Appendix]{SjoesZworski02Monodromy} for a proof in the semiclassical case.
We relate the subprincipal symbol to the leading geometric symbol.
Equip the half-density bundle $\halfdens\ueber X$ with the inner
product $(u|v)=u\cdot\bar{v}/\dvol_X$, where the operations on the right
are in the sense of densities.
The connection given by $\nabla^{\halfdens} \dvol_X^{1/2} =0$ is metric
with respect to the Hermitian structure of $\halfdens$.

\begin{cor}
\label{cor-subprinc-symbol}
Let $A\in\Pskm{k}{m}(X;\halfdens)$.
The leading symbol of $A$ equals that of the corresponding scalar operator
$\tilde{A}\in\Pskm{k}{m}(X)$ which is given by
$\tilde{A}u= \dvol_X^{-1/2} A(u\dvol_X^{1/2})$.
If the geometric symbol $a$ of $A$ is classical,
$a\sim\sum_{j\geq 0} \h^{j-k}a_j$, $a_j\in\Sym{m-j}$,
then $\h^{-k} a_0$ is the principal symbol of $A$, and
\[
\subprinc{a}= \h^{1-k}(a_1 + i \verderiv a_0.\horderiv a_0/2)
\]
is its subprincipal symbol.
\end{cor}
\begin{proof}
Consider the multiplication operator $\dvol_X^{1/2}\in\Pskm{0}{0}(X;\C,\halfdens)$.
The $\Hom(\C,\halfdens)$-valued symbol $\pi^* \dvol_X^{1/2}$ is the
leading symbol of this operator.
Note that its horizontal and vertical derivatives vanish.
The equality of the leading symbols of $A$ and $\tilde{A}$
now follows from Proposition~\ref{geocalc-compos}.

Let $a^U$ denote the local symbol of $A$ in a geodesic
coordinate chart $U$ centered at a given point $x$.
We use normal coordinates centered at $x$.
Assume $a$ classical, $\h^k a= a_0+\h a_1+\bigoh(\h^2)$.
Then $a^U$ is classical, and $\h^k a^U= a_0+\bigoh(\h)$.
Moreover, it follows from Lemma~\ref{def-geom-Oph} that
$\h^k a^U= a_0+\h a_1+\bigoh(\h^2)$ at $x$.
The subprincipal symbol equals, by definition,
$\h^{1-k}(a_1 +2^{-1}i\sum_j \partial^2 a_0/\partial x_j \partial\xi_j)$.
The horizontal derivative in the $j$-th coordinate direction equals,
at $x$, the partial derivative with respect to $x_j$.
The formula for the subprincipal symbol follows.
\end{proof}


\def\cprime{$'$}

\end{document}